\documentclass[article,12pt]{amsart}
\usepackage{amsfonts}
\usepackage{amsmath}
\usepackage{graphicx}
\usepackage{color}
\usepackage{epsfig}
\usepackage{float}
\usepackage{multirow}

\newcommand{\dC}{\mathbb{C}}

\newcommand{\dR}{\mathbb{R}}

\newcommand{\dP}{\mathbb{P}}
\newcommand{\cA}{\mathcal{A}}
\newcommand{\cB}{\mathcal{B}}

\newcommand{\cH}{\mathcal{H}}
\newcommand{\cN}{\mathcal{N}}

\newcommand{\cR}{\mathcal{R}}

\newcommand{\cO}{\mathcal{O}}

\newcommand{\rI}{\mathrm{I}}
\newcommand{\rJ}{\mathrm{J}}

\newcommand{\cZ}{\mathcal{Z}}

\newcommand{\veps}{\varepsilon}
\newcommand{\wh}{\widehat}
\newcommand{\wt}{\widetilde}

\newcommand{\ind}{\mbox{1}\kern-.25em \mbox{I}}
\font\calcal=cmsy10 scaled\magstep1
\def\build#1_#2^#3{\mathrel{\mathop{\kern 0pt#1}\limits_{#2}^{#3}}}
\def\liml{\build{\longrightarrow}_{}^{{\mbox{\calcal L}}}}

\def\videbox{\mathbin{\vbox{\hrule\hbox{\vrule height1ex \kern.5em
\vrule height1ex}\hrule}}}
\def\demend{\hfill $\videbox$\\}

\numberwithin{equation}{section}
\theoremstyle{plain}
\newtheorem{lem}{Lemma}[section]
\newtheorem{thm}{Theorem}[section]
\newtheorem{rem}{Remark}[section]
\newtheorem{cor}{Corollary}[section]

\email{Bernard.Bercu@math.u-bordeaux1.fr}
\email{Bruno.Portier@insa-rouen.fr}
\email{Vvazquez@fcfm.buap.mx}

\begin{document}
\title[Asymptotic behavior of the Durbin-Watson statistic for ARX processes]
{On the asymptotic behavior of the Durbin-Watson statistic for ARX processes in adaptive tracking}
\author{Bernard Bercu}
	\address{ Universit\'e Bordeaux 1, Institut de Math\'ematiques de Bordeaux,
	UMR 5251, 351 cours de la lib\'eration, 33405 Talence cedex, France.}

\author{Bruno Portier}
\address{D\'epartement de G\'enie Math\'ematiques, Laboratoire de Math\'ematiques, INSA de Rouen, LMI-EA 3226, place Emile Blondel, BP 08, 76131 Mont-Saint-Aignan cedex, France}

\author{Victor Vazquez}
\address{ Universidad Aut\'onoma de Puebla, Facultad de Ciencias Fisico
Matem\'aticas, Avenida San Claudio y Rio Verde, 72570 Puebla, Mexico.}

\begin{abstract}
A wide literature is available on the asymptotic behavior of the Durbin-Watson statistic for autoregressive models.
However, it is impossible to find results on the Durbin-Watson statistic 
for autoregressive models with adaptive control. Our purpose is to fill the gap by establishing 
the asymptotic behavior of the Durbin Watson statistic for ARX models
in adaptive tracking. On the one hand, we show the almost sure convergence 
as well as the asymptotic normality of the least squares estimators of the unknown
parameters of the ARX models. On the other hand, we establish the almost sure convergence 
of the Durbin-Watson statistic and its asymptotic normality. Finally, we propose a bilateral statistical
test for residual autocorrelation in adaptive tracking.
\end{abstract}

\maketitle

%%%%%%%%%%%%%%%%%%%%%%%%%%%%%%%%%%%%%%%%%%%%%%%%%%%%%%%%%%%%%%%%%%%%%%%%%%%%%%%%%%%%%%%%%
\section{Introduction and Motivation}
\label{SectionIM}
%%%%%%%%%%%%%%%%%%%%%%%%%%%%%%%%%%%%%%%%%%%%%%%%%%%%%%%%%%%%%%%%%%%%%%%%%%%%%%%%%%%%%%%%%
\vspace{2ex}
The Durbin-Watson statistic was introduced in the pioneer works of Durbin and Watson \cite{DW1}, \cite{DW2}, \cite{DW3},
in order to detect the presence of a first-order autocorrelated driven noise in linear regression models.
A wide literature is available on the asymptotic behavior of the Durbin-Watson statistic for linear regression 
models and it is well-known that the statistical test based on the Durbin-Watson statistic performs pretty well
when the regressors are independent random variables. 
However, as soon as the regressors are lagged dependent variables, 
which is of course the most attractive case,
its widespread use in inappropriate situations may lead to bad conclusions. 
More precisely, it was observed by Malinvaud \cite{Malinvaud} and Nerlove and Wallis \cite{Nerlove} that
the Durbin-Watson statistic may be asymptotically biased if
the model itself and the driven noise are governed by first-order autoregressive processes. 
In order to prevent this misuse, Durbin \cite{Durbin} proposed a redesigned alternative test in the particular case 
of the first-order autoregressive process previously investigated in \cite{Malinvaud}, \cite{Nerlove}. 
More recently, Stocker \cite{Stocker} provided substantial improvements
in the study of the asymptotic behavior of the Durbin-Watson statistic
resulting from the presence of a first-order autocorrelated noise.
We also refer the reader to Bercu and Proia \cite{BPDW} for a recent sharp analysis on the
asymptotic behavior of the Durbin-Watson statistic via a martingale approach.
\ \vspace{1ex}\\
Moreover, far as the authors know, there are no established results on the Durbin-Watson statistic 
for autoregressive models with exogenous control. 
Therefore, our purpose is to investigate the asymptotic behavior of the Durbin-Watson statistic for the 
$\text{ARX}(p,q)$ processes where $p \geq 1$ and $q\geq 0$. 
We focus our attention on the $\text{ARX}(p,0)$ process, given for all $n\geq 0$, by
\begin{equation} 
\label{ARXP}
X_{n+1}=\sum_{k=1}^p\theta_k X_{n-k+1}+U_n+\varepsilon_{n+1}
\end{equation}
in which the driven noise $(\varepsilon_{n})$ follows the first-order autoregressive process
\vspace{2ex}
\begin{equation} 
\label{COR}
\varepsilon_{n+1}=\rho\varepsilon_{n}+V_{n+1}.
\vspace{2ex}
\end{equation}
We assume that the serial autocorrelation parameter satisfies $|\rho|<1$ and the initial values $X_0$, 
$\varepsilon_0$ and $U_0$ may be arbitrarily chosen. In all the sequel, we also assume that $(V_n)$ is a martingale difference
sequence adapted to the filtration $\mathbb{F} = (\mathcal{F}_n)$ where $\mathcal{F}_n$ stands
for the $\sigma$-algebra of the events occurring up to time $n$. Moreover, we suppose that, for all $n \geq 0$, 
$\mathbb{E}\left[V_{n+1}^2|\mathcal{F}_n\right]=\sigma^2$ a.s. with $\sigma^2>0$. 
Denote by $\theta$ the unknown parameter of equation (\ref{ARXP})
%$$\theta = \Bigl(\theta_1, \theta_2, \ldots, \theta_p \Bigr)^{t}. $$

$$\theta = \begin{pmatrix}
\ \theta_1 \ \\
\ \theta_2 \ \\
\vdots \\
\theta_p
\end{pmatrix}.
$$

Our goal is to deal simultaneously with three objectives. The first one is to propose an efficient procedure in order to estimate
the unknown parameters $\theta$ and $\rho$ of the $\text{ARX}(p,0)$ process given by \eqref{ARXP} and \eqref{COR}.
The second one is to regulate the dynamic of the process $(X_n)$ by forcing $X_n$ to track step by step a predictable reference trajectory $(x_n)$. 
This second objective can be achieved by use of an appropriate version of the adaptive tracking control
proposed by Astr$\ddot{\mbox{o}}$m and Wittenmark \cite{Astrom}. Finally, our last objective is to establish the aymptotic properties of 
the Durbin-Watson statistic in order 
to propose a bilateral test on the serial parameter $\rho$.
\ \vspace{2ex} \\
The paper is organized as follows. Section 2 is devoted to the parameter estimation procedure and the suitable choice
of stochastic adaptive control. In Section 3, we establish the almost sure convergence of the least squares estimators of $\theta$ and $\rho$. 
The asymptotic normality of our estimates are given in Section 4.
We shall be able in Section 5 to prove the almost sure convergence of the Durbin-Watson statistic
as well as its asymptotic normality, which will lead us to propose a bilateral statistical test for 
residual autocorrelation. Some numerical simulations are provided in
Section 6. Finally, all technical proofs are postponed in the
Appendices.

%%%%%%%%%%%%%%%%%%%%%%%%%%%%%%%%%%%%%%%%%%%%%%%%%%%%%%%%%%%%%%%%%%%%%%%%%%%%%%%%%%%%%%%%%
\section{Estimation and Adative Control}
\label{SectionEC}
%%%%%%%%%%%%%%%%%%%%%%%%%%%%%%%%%%%%%%%%%%%%%%%%%%%%%%%%%%%%%%%%%%%%%%%%%%%%%%%%%%%%%%%%%

Relation \eqref{ARXP} can be rewritten as
\begin{equation} 
\label{ARXPCF}
X_{n+1}=\theta^t \varphi_{n}+U_n+\varepsilon_{n+1}
\end{equation}
where
$$\varphi_n = \begin{pmatrix}
\  X_{n} \ \\
\ X_{n-1} \ \\
\vdots \\
 X_{n-p+1}
 \vspace{1ex}
\end{pmatrix}.
$$
A naive strategy to regulate the dynamic of the process $(X_n)$ is to make use of the
Astr$\ddot{\mbox{o}}$m-Wittenmark \cite{Astrom} adaptive tracking control
$$
U_n = x_{n+1}- \wh{\theta}_n^{\, t} \varphi_n
$$
where $\wh{\theta}_n$ stands for the least squares estimator of $\theta$.
Unfortunately, we can show that this strategy leads to biased estimation
of the parameters $\theta$ and $\rho$. 
This is due to the fact that $(\varepsilon_{n})$ is not a white noise but the first-order autoregressive process given by 
\eqref{COR}. Consequently, it is necessary to adopt a more appropriate strategy which means a more suitable choice
for the adaptive control $U_n$ in \eqref{ARXPCF}.
\ \vspace{1ex} \\
The construction of our control law is as follows. Starting from (\ref{ARXP}) together with (\ref{COR}),
we easily deduce that the process $(X_n)$ satisfies the fundamental
$\text{ARX}(p+1,1)$ equation given, for all $n \geq 1$, by
\begin{eqnarray} 
X_{n+1}&=&(\theta_1+\rho)X_n+(\theta_2-\rho\theta_1)X_{n-1}+\cdots+(\theta_p-\rho\theta_{p-1})X_{n-p+1}\nonumber
 \\
& &-\rho\theta_pX_{n-p}+U_n-\rho U_{n-1} +V_{n+1} 
\label{EQFONDA}
\end{eqnarray}
which can be rewritten as
\begin{equation}
\label{ARXPNEW}
X_{n+1}=\vartheta^t\Phi_n+U_n+V_{n+1}
\end{equation}
where the new parameter $\vartheta \in \dR^{p+2}$ is defined as
\begin{equation}
\label{NEWPARA}
\vartheta = 
\begin{pmatrix}
\  \theta \ \\
\ 0 \ \\
0
\end{pmatrix}
-
\rho
\begin{pmatrix}
 -1 \ \\
\ \, \theta \ \\
\ 1
\end{pmatrix}
\end{equation}
and the new regression vector $\Phi_n$ is given by
\begin{equation*}
\Phi_n = 
\begin{pmatrix}
\  \varphi_n \ \\
\ X_{n-p} \ \\
U_{n-1}
\end{pmatrix}.
\end{equation*}
The original idea of this paper is to control the model \eqref{ARXPCF} using the adaptive control 
associated with the model \eqref{ARXPNEW} in order to a posteriori estimate the parameters  $\theta$
$\rho$ via the estimator of the parameter $\vartheta$.
We shall now focus our attention on the estimation of the unknown parameter $\vartheta$. We propose to make use of the 
least squares estimator which satisfies, for all $n\geq 0$, 
\begin{equation}  
\label{LSVARTHETA}
\wh{\vartheta}_{n+1}=
\wh{\vartheta}_{n}+S_{n}^{-1}\Phi_{n}
\Bigl(X_{n+1}-U_{n}-\wh{\vartheta}_{n}^{\, t}\Phi_n\Bigr)
\end{equation}
where the initial value $\wh{\vartheta}_{0}$ may be arbitrarily chosen and 
\begin{equation*}
S_{n}=\sum_{k=0}^{n}\Phi_{k}\Phi_{k}^{t}+\rI_{p+2}
\end{equation*}
where the identity matrix $\rI_{p+2}$ is added in order
to avoid useless invertibility assumption. On the other hand, we are concern with 
the crucial choice of the adaptive control $U_n$. The role played by $U_n$ is to regulate the dynamic of the process 
$(X_n)$ by forcing $X_n$ to track step by step a predictable reference trajectory $(x_n)$. In order
to control the dynamic of $(X_n)$ given by \eqref{ARXP}, we propose to make use of the 
Astr$\ddot{\mbox{o}}$m-Wittenmark \cite{Astrom} adaptive
tracking control associated with \eqref{ARXPNEW} and given, for all $n \geq 0$, by
\begin{equation}  
\label{CONTROL}
U_n = x_{n+1}-\wh{\vartheta}_n^{\,t}\,\Phi_n.
\end{equation}
This suitable choice of $U_n$ will allow us to control the dynamic of the process
\eqref{ARXPCF} while maintaining the optimality of the tracking
and then estimate without bias the parameters $\theta$ and $\rho$.
In all the sequel, we assume that the reference trajectory $(x_n)$ satisfies 
\begin{equation}  
\label{CT}
\sum_{k=1}^{n} x_{k}^{2} =o(n) \hspace{1cm} \text{a.s.}
\end{equation}

%%%%%%%%%%%%%%%%%%%%%%%%%%%%%%%%%%%%%%%%%%%%%%%%%%%%%%%%%%%%%%%%%%%%%%%%%%%%%%%%%%%%%%%%%
\section{Almost sure convergence}
\label{SectionASC}
%%%%%%%%%%%%%%%%%%%%%%%%%%%%%%%%%%%%%%%%%%%%%%%%%%%%%%%%%%%%%%%%%%%%%%%%%%%%%%%%%%%%%%%%%

All our asymptotic analysis relies on the following keystone lemma.
First of all, let $L$ be the identity matrix of order $p+1$ and denote by $H$ the positive real number
\begin{equation}
\label{DEFH}
H=\sum_{k=1}^p (\theta_k+\rho^k)^2+\frac{\rho^{2(p+1)}}{1-\rho^2}.
\end{equation}
In addition, for $1\leq k \leq p$, let $K_k=-(\theta_k+\rho^k)$ and denote by
$K$ the line vector 
\begin{equation}
\label{DEFK}
K=\Bigl(0,K_1,K_2,\ldots,K_p\Bigr).
\end{equation}
Moreover, let $\Lambda$ be the symmetric square matrix of order $p+2$,
\begin{equation}  
\label{DEFLAMBDA}
\Lambda=\left( 
\begin{array}{cc}
L & K^t \\ 
K &H
\end{array}
\right).
\end{equation}

\begin{lem}
\label{L-CVGSN}
Assume that $(V_n)$ has a finite conditional moment of order $>2$. Then, we have
\begin{equation}
\label{CVGSN}
\lim_{n\rightarrow \infty} \frac{1}{n}S_n=\sigma^2 \Lambda \hspace{1cm} \text{a.s.}
\end{equation}
where the limiting matrix $\Lambda$ is given by \eqref{DEFLAMBDA}.
In addition, as soon as the correlation parameter $\rho \neq 0$, the matrix $\Lambda$ is invertible and
\begin{equation}
\label{INVLAMBDA}
\Lambda^{-1}=\frac{1-\rho^2}{\rho^{2(p+1)}}\left( 
\begin{array}{cc}
SL+K^tK & -K^t \\ 
-K & 1
\end{array}
\right)
\end{equation}
where $S= H- || K ||^2$ is the Schur complement of $L$ in $\Lambda$, 
\begin{equation}
\label{DEFS}
S= \frac{\rho^{2(p+1)}}{1-\rho^2}.
\end{equation}
\end{lem}
\noindent{\bf Proof.}
The proof is given in Appendix\,A. \demend
\vspace{-2ex} 
\begin{rem}
As $L$ is the identity matrix of order $p+1$, we clearly have
\begin{equation*}
\det(\Lambda)=\frac{\rho^{2(p+1)}}{1-\rho^2}.
\end{equation*} 
Consequently, as long as $\rho \neq 0$, $\det(\Lambda) \neq 0$ which of course implies that
the matrix $\Lambda$ is invertible.  The identity \eqref{INVLAMBDA} comes from the block matrix
inversion formula given e.g. by Horn and Johnson \cite{Horn}, page 18.
\end{rem}

The almost sure properties of the least squares estimator $\wh{\vartheta}_{n}$  of $\vartheta$ are as follows.

\begin{thm} 
\label{T-ASCVGVARTHETA}
Assume that the serial correlation parameter $\rho\neq 0$ and that $(V_n)$ has a finite conditional moment of order $>2$. 
Then, $\wh{\vartheta}_{n}$ converges almost surely to $\vartheta$, 
\begin{equation}  
\label{ASCVGVARTHETA}
\parallel \wh{\vartheta}_{n}-\vartheta \parallel^{2}= \mathcal{O} 
\left( \frac{\log n}{n} \right) 
\hspace{0.5cm}\text{a.s.}
\end{equation}
\end{thm}
\noindent{\bf Proof.}
The proof is given in Appendix\,A. \demend
\ \vspace{-1ex} \\
We shall now explicit the estimators of $\theta$ and $\rho$ and their convergence results.
It follows from \eqref{NEWPARA} that
\begin{equation}
\label{INIPARA}
\begin{pmatrix}
\ \theta \ \\
\ \rho \
\end{pmatrix}
 = \Delta \vartheta
\end{equation}
where $\Delta$ is the rectangular matrix of size $(p+1)\!\times\!(p+2)$ given by
\begin{equation}
\label{DEFDELTA}
\Delta=
\begin{pmatrix}
1 & 0 & \cdots & \cdots & \cdots & 0 & 1 \\ 
\rho & 1 & 0 & \cdots & \cdots & 0 & \rho \\ 
\rho^2 & \rho & 1 & 0 & \cdots & 0 & \rho^2 \\ 
\cdots & \cdots & \cdots & \cdots & \cdots & \cdots & \cdots \\ 
\rho^{p-1} & \rho^{p-2} & \cdots & \rho & 1 & 0 & \rho^{p-1} \\ 
0 & 0 & \cdots & \cdots & \cdots & 0 & -1
\end{pmatrix}.
\vspace{1ex}
\end{equation}
Consequently, a natural choice to estimate the initial parameters $\theta$ and $\rho$ is to make use of
\begin{equation}
\label{LSTHETARHO}
\begin{pmatrix}
\ \wh{\theta}_{n} \ \\
\ \wh{\rho}_{n} \
\end{pmatrix}
 = \wh{\Delta}_{n} \wh{\vartheta}_{n}
\end{equation}
where $\wh{\rho}_{n}$ is simply the opposite of the last coordinate of $\wh{\vartheta}_{n}$ and
\begin{equation}
\label{DEFDELTAN}
\wh{\Delta}_{n}=
\begin{pmatrix}
1 & 0 & \cdots & \cdots & \cdots & 0 & 1 \\ 
\wh{\rho}_{n} & 1 & 0 & \cdots & \cdots & 0 & \wh{\rho}_{n} \\ 
\wh{\rho}_{n}^{\,2} & \wh{\rho}_{n} & 1 & 0 & \cdots & 0 & \wh{\rho}_{n}^{\,2}  \\ 
\cdots & \cdots & \cdots & \cdots & \cdots & \cdots & \cdots \\ 
\wh{\rho}_{n}^{\,p-1}  & \wh{\rho}_{n}^{\,p-2} & \cdots & \wh{\rho}_{n} & 1 & 0 & \wh{\rho}_{n}^{\,p-1}  \\ 
0 & 0 & \cdots & \cdots & \cdots & 0 & -1
\end{pmatrix}
.
\end{equation}

\begin{cor}
\label{C-ASCVGTHETARHO}
Assume that the serial correlation parameter $\rho\neq 0$ and that $(V_n)$ has a finite conditional moment of order $>2$. 
Then, $\wh{\theta}_{n}$ and $\wh{\rho}_{n}$ both converge almost surely to $\theta$ and $\rho$,
\begin{equation}  
\label{ASCVGTHETA}
\parallel \wh{\theta}_{n}-\theta \parallel^{2}= \mathcal{O} 
\left( \frac{\log n}{n} \right) 
\hspace{0.5cm}\text{a.s.}
\end{equation}
\begin{equation}  
\label{ASCVGRHO}
 (\wh{\rho}_{n}-\rho)^{2}= \mathcal{O} 
\left( \frac{\log n}{n} \right) 
\hspace{0.5cm}\text{a.s.}
\end{equation}
\end{cor}

\begin{proof}
One can immediately see from \eqref{INIPARA}  that the last component of
the vector $\vartheta$ is $-\rho$. The same is true for the estimator
$\wh{\rho}_{n}$ of $\rho$. Consequently, we deduce
from \eqref{ASCVGVARTHETA} that $\wh{\rho}_{n}$ converges a.s. to $\rho$ with the almost sure rate of convergence
given by \eqref{ASCVGRHO}. Therefore, we obtain from \eqref{DEFDELTA} and \eqref{DEFDELTAN} that
$$
\parallel \wh{\Delta}_n - \Delta \parallel^2= \mathcal{O} 
\left( \frac{\log n}{n} \right) 
\hspace{0.5cm}\text{a.s.}
$$
which ensures via  \eqref{ASCVGVARTHETA} and \eqref{LSTHETARHO} that $\wh{\theta}_{n}$ 
converges a.s. to $\theta$ with the almost sure rate of convergence
given by \eqref{ASCVGTHETA}.
\end{proof}

%%%%%%%%%%%%%%%%%%%%%%%%%%%%%%%%%%%%%%%%%%%%%%%%%%%%%%%%%%%%%%%%%%%%%%%%%%%%%%%%%%%%%%%%%
\section{Asymptotic Normality}
\label{SectionAN}
%%%%%%%%%%%%%%%%%%%%%%%%%%%%%%%%%%%%%%%%%%%%%%%%%%%%%%%%%%%%%%%%%%%%%%%%%%%%%%%%%%%%%%%%%

This Section is devoted to the asymptotic normality of the estimators associated with $\theta$ and $\rho$
which is obtained from the asymptotic normality of the least squares estimator $\wh{\vartheta}_n$ of $\vartheta$.

\begin{thm} 
\label{T-CLTVARTHETA}
Assume that the serial correlation parameter $\rho\neq 0$ and that $(V_n)$ has a finite conditional moment of order $>2$. 
In addition, suppose that $(x_n)$ has the same almost sure regularity as $(V_n)$.
Then, we have
\begin{equation}  
\label{CLTVARTHETA}
\sqrt{n}(\wh{\vartheta}_{n}-\vartheta )\liml
\cN(0,\Lambda^{-1} ) 
\end{equation}
where the matrix $\Lambda^{-1}$ is given by \eqref{INVLAMBDA}.
\end{thm}

In order to provide the joint asymptotic normality of the estimators of $\theta$ and $\rho$, denote, for all 
$1\leq k \leq p-1$, 
$$
\xi_k=\sum_{i=1}^{k}\rho^{k-i}\theta_i
$$
and let $\nabla$ be the rectangular matrix of size $(p+1)\!\times\!(p+2)$ given by
\begin{equation}
\label{DEFNABLA}
\nabla=
\begin{pmatrix}
1 & 0 & \cdots & \cdots & \cdots & 0 & 1 \\ 
\rho & 1 & 0 & \cdots & \cdots & 0 & \rho - \xi_1 \\ 
\rho^2 & \rho & 1 & 0 & \cdots & 0 & \rho^2 -  \xi_2\\ 
\cdots & \cdots & \cdots & \cdots & \cdots & \cdots & \cdots \\ 
\rho^{p-1} & \rho^{p-2} & \cdots & \rho & 1 & 0 & \rho^{p-1} -\xi_{p-1}\\ 
0 & 0 & \cdots & \cdots & \cdots & 0 & -1
\end{pmatrix}.
\end{equation}

\begin{cor}
\label{C-CLTTHETARHO}
Assume that the serial correlation parameter $\rho\neq 0$ and that $(V_n)$ has a finite conditional moment of order $>2$. 
In addition, suppose that $(x_n)$ has the same almost sure regularity as $(V_n)$.
Then, we have
\begin{equation}
\label{CLTTHETARHO}
\sqrt{n}
\begin{pmatrix}
\ \wh{\theta}_{n} - \theta \ \\
\ \wh{\rho}_{n} - \rho \
\end{pmatrix}
\liml \cN(0, \Sigma )
\end{equation}
where $\Sigma =\nabla \Lambda^{-1} \nabla^{t}$. In particular,
\begin{equation}
\label{CLTRHO}
\sqrt{n}(\widehat{\rho}_{n}-\rho)\liml
\cN\left(0,\frac{1-\rho^2}{\rho^{2(p+1)}}\right).
\end{equation}
\end{cor}

\noindent{\bf Proof.}
The proof is given in Appendix\,B. \demend
%%

%%%%%%%%%%%%%%%%%%%%%%%%%%%%%%%%%%%%%%%%%%%%%%%%%%%%%%%%%%%%%%%%%%%%%%%%%%%%%%%%%%%%%%%%%
\section{On the Durbin Watson statistic}
\label{SectionDW}
%%%%%%%%%%%%%%%%%%%%%%%%%%%%%%%%%%%%%%%%%%%%%%%%%%%%%%%%%%%%%%%%%%%%%%%%%%%%%%%%%%%%%%%%%

We now investigate the asymptotic behavior of the Durbin-Watson statistic \cite{DW1}, \cite{DW2}, \cite{DW3} given, for all $n\geq 1$, by 

\begin{equation}
\label{DEFDN}
\wh{D}_n=\frac{\sum_{k=1}^{n}\left(\widehat{\varepsilon}_k-\widehat{\varepsilon}_{k-1}\right)^2}{\sum_{k=0}^{n}\widehat{\varepsilon}_{k}^{\,2}}
\end{equation}
where the residuals $\widehat{\varepsilon}_{k}$ are defined, for all $1\leq k\leq n$, by
\begin{equation}
\label{RESIDUALN}
\widehat{\varepsilon}_{k}=X_k-U_{k-1}-\wh{\theta}_{n}^{\,t}\varphi_{k-1}
\end{equation}
with $\wh{\theta}_{n}$ given by \eqref{LSTHETARHO}.
The initial value $\widehat{\varepsilon}_{0}$ may be arbitrarily chosen and we take $\widehat{\varepsilon}_{0}=X_{0}$. 
One can observe that it is also possible to estimate the serial correlation parameter $\rho$ 
by the least squares estimator
\begin{equation} 
\label{NEWLSRHO}
\overline{\rho}_n=\frac{\sum_{k=1}^{n}\widehat{\varepsilon}_k\widehat{\varepsilon}_{k-1}}{\sum_{k=1}^{n}\widehat{\varepsilon}_{k-1}^{\,2}}
\end{equation}
which is the natural estimator of $\rho$ in the autoregressive framework without control.
The Durbin-Watson statistic $\wh{D}_n$ is related to $\overline{\rho}_n$ by the linear relation
\begin{equation}
\label{RELRHODN}
\wh{D}_n=2(1-\overline{\rho}_n)+\zeta_n
\end{equation}
where the remainder term $\zeta_n$ plays a negligeable role. 
The almost sure properties of $\wh{D}_n$ and $\rho_n$ are as follows.

\begin{thm}
\label{T-ASCVGDW}
Assume that the serial correlation parameter $\rho \! \neq\! 0$ and that $(V_n)$ has a finite conditional moment of order $>2$. 
Then, $\overline{\rho}_n$ converges almost surely to $\rho$,
\begin{equation}  
\label{ASCVGNRHO}
( \overline{\rho}_{n}-\rho)^{2}= \mathcal{O} \left( \frac{\log n}{n} \right) 
\hspace{0.5cm}\text{a.s.}
\end{equation}
In addition, $\wh{D}_{n}$ converges almost surely to $D=2(1-\rho)$.
Moreover, if $(V_n)$ has a finite conditional moment of order $>4$, we also have
\begin{equation}  
\label{ASCVGDW}
\Bigl( \wh{D}_{n}-D \Bigr)^{2}= \mathcal{O} \left( \frac{\log n}{n} \right) 
\hspace{0.5cm}\text{a.s.}
\end{equation}
\end{thm}

Our next result deals with the asymptotic normality of the Durbin-Watson statistic.
For that purpose, it is necessary to introduce some notations. Denote 
\begin{equation}
\label{DEFALPHABETA}
\alpha 
= \begin{pmatrix}
\  1 \ \\
-\theta_{1} \ \\
\vdots \\
-\theta_{p} \ \\
\!\!-1
 \vspace{1ex}
\end{pmatrix}
\hspace{1cm} \text{and} \hspace{1cm} 
\beta 
= \begin{pmatrix}
  1 \\
 \rho  \\
\vdots \\
\rho^{p-1}  \\
 0
 \vspace{1ex}
\end{pmatrix}.
\end{equation}
In addition, let 
\begin{equation}
\label{DEFGAMMA}
\gamma= \Lambda \alpha + (1- \rho^2) \nabla^{t} \beta.
\end{equation}

\begin{thm}
\label{T-CLTDW}
Assume that the serial correlation parameter $\rho \! \neq\! 0$ and that $(V_n)$ has a finite conditional moment of order $>2$. 
In addition, suppose that $(x_n)$ has the same almost sure regularity as $(V_n)$.
Then, we have
\begin{equation}
\label{CLTRHOO}
\sqrt{n}(\overline{\rho}_{n}-\rho)\liml
\cN\left(0,\tau^2\right)
\end{equation}
where the asymptotic variance $\tau^2= (1 - \rho^2)^2 \gamma^{t}  \Lambda^{-1} \gamma$.
Moreover, if $(V_n)$ has a finite conditional moment of order $>4$, we also have
\begin{equation}
\label{CLTDW}
\sqrt{n}(\wh{D}_{n}-D)\liml
\cN\left(0,4\tau^2\right)
\end{equation}
\end{thm}
\noindent{\bf Proof.}
The proofs are given in Appendix\,C. \demend

\vspace{-2ex} 
\begin{rem}
It follows from \eqref{INVLAMBDA} together with tedious but straighforward calculations that
for all $p\geq 1$,
\begin{eqnarray}
\label{EXPTAU}
\tau^2 &=& \frac{(1-\rho^2)}{\rho^{2(p+1)}}
\bigg[\rho^{2(p+1)} \Bigl(4-(4p+3)\rho^{2p} + 4p \rho^{2(p+1)} - \rho^{2(2p+1)}\Bigr) \notag \\
& & \hspace{1cm} + \Bigl(1 - (p+1)\rho^{2p} + (p-1) \rho^{2(p+1)}\Bigr)^2\bigg].
\end{eqnarray}
For example, in the particular case $p=1$, we obtain that
\begin{equation}
\label{EXPTAU1}
\tau^2 = \frac{(1-\rho^2)}{\rho^{4}}
\bigg(1 - 4 \rho^2 + 8 \rho^{4} -7 \rho^{6} +4 \rho^{8}- \rho^{10}\bigg).
\end{equation}
Moreover, it is not hard to see by a convexity argument that we always have for all $p\geq 1$,
$$
\tau^2 \leq \frac{1-\rho^2}{\rho^{2(p+1)}}.
$$
In other words, the  least squares estimator $\overline{\rho}_n$ performs better than $\wh{\rho}_n$ for the estimation
of $\rho$. It means that a statistical test procedure built on the Durbin-Watson statistic should be really powerful.
\end{rem}

We are now in the position to propose our new  bilateral statistical test built on the Durbin-Watson statistic $\wh{D}_n$. First of all, 
we shall not investigate the case $\rho=0$ since our approach is only of interest for ARX processes
where the driven noise is given by a first-order autoregressive process. For a given value $\rho_0$ such that $|\rho_0|< 1$
and $\rho_0 \neq 0$, we wish to test whether or not the serial correlation parameter is equal to $\rho_0$. It means that
we wish to test

$$ \cH_0\,:\,`` \rho = \rho_0" \hspace{1cm} \text{against}\hspace{1cm}\cH_1\,:\,`` \rho \neq \rho_0".
\vspace{2ex} $$
According to Theorem \ref{T-ASCVGDW}, we have under the null hypothesis $\cH_0$
$$
\lim_{n \rightarrow \infty} \wh{D}_n=D_{0}
\hspace{1cm} \text{a.s.}
$$
where $D_0=2(1-\rho_0)$. In addition, we clearly have from \eqref{CLTDW} that under $\cH_0$
\begin{equation}
\label{CHICLTDW}
\frac{n}{4\tau^2} \left( \wh{D}_{n} - D_{0} \right)^2 \liml \chi^2
\end{equation}
where $\chi^2$ stands for a Chi-square distribution with one degree of freedom. Via \eqref{EXPTAU},
an efficient strategy to estimate the asymptotic variance $\tau^2$ is to make use of
\begin{eqnarray}
\label{DEFGAMMATAUNEW}
\wh{\tau}_n^{\,2} &=& \frac{(1-\overline{\rho}_n^{\,2})}{\overline{\rho}_n^{\,2(p+1)}}
\bigg[\overline{\rho}_n^{\,2(p+1)} \Bigl(4-(4p+3)\overline{\rho}_n^{\,2p} + 4p \overline{\rho}_n^{\,2(p+1)} - \overline{\rho}_n^{\,2(2p+1)}\Bigr) \notag \\
& & \hspace{1cm} + \Bigl(1 - (p+1)\overline{\rho}_n^{\,2p} + (p-1) \overline{\rho}_n^{\,2(p+1)}\Bigr)^2\bigg].
\end{eqnarray}
Therefore, our new  bilateral statistical test relies on the following result.

\begin{thm}
\label{T-DWTEST}
Assume that the serial correlation parameter $\rho \! \neq\! 0$ and that $(V_n)$ has a finite conditional moment of order $>4$. 
In addition, suppose that $(x_n)$ has the same almost sure regularity as $(V_n)$.
Then, under the null hypothesis $\cH_0\,:\,`` \rho = \rho_0"$,
\begin{equation}
\label{DWTESTH0}
\frac{n}{4\wh{\tau}_n^{\, 2}} \left( \wh{D}_{n} - D_{0} \right)^2 \liml \chi^2
\end{equation}
where $\chi^2$ stands for a Chi-square distribution with one degree of freedom. In addition, 
under the alternative hypothesis $\cH_1\,:\,`` \rho \neq \rho_0"$,
\begin{equation}
\label{DWTESTH1}
\lim_{n\rightarrow \infty}  \frac{n}{4\wh{\tau}_n^{\, 2}} \left( \wh{D}_{n} - D_{0}  \right)^2 = + \infty \hspace{1cm} \textnormal{a.s.}
\end{equation}
\end{thm}
\noindent{\bf Proof.}
The proof is given in Appendix\,C. \demend

From a practical point of view, for a significance level $\alpha$ where $0<\alpha <1$, the acceptance and rejection regions are given by 
$\cA= [0, a_{\alpha}]$ and $\cR =] a_{\alpha}, +\infty [$
where $a_{\alpha}$ stands for the $(1-\alpha)$-quantile of the Chi-square distribution with one degree of freedom.
The null hypothesis $\cH_0$ will be accepted if 
$$
\frac{n}{4\wh{\tau}_n^{\, 2}} \left( \wh{D}_{n} - D_{0} \right)^2 \leq a_{\alpha},
$$
and will be rejected otherwise.

%%%%%%%%%%%%%%%%%%%%%%%%%%%%%%%%%%%%%%%%%%%%%%%%%%%%%%%%%%%%%%%%%%%%%%%%%%%%%%%%%%%%%%%%%
\section{Numerical Experiments}
\label{SectionNS}
%%%%%%%%%%%%%%%%%%%%%%%%%%%%%%%%%%%%%%%%%%%%%%%%%%%%%%%%%%%%%%%%%%%%%%%%%%%%%%%%%%%%%%%%%

The purpose of this section is to provide some numerical experiments in order to illustrate our
main theoretical results. In order to keep this section brief, we shall only consider the  $\text{ARX}(p,0)$
process $(X_n)$ given by (\ref{ARXP}) in the particular cases $p=1$ and $p=2$,
where the driven noise $(\varepsilon_n)$ satisfies (\ref{COR}). Moreover, for the sake of simplicity, the reference 
trajectory $(x_n)$ is chosen to be identically zero and $(V_n)$ is a Gaussian white noise with $\mathcal{N}(0,1)$ distribution. 
Finally our numerical simulations are based on $500$ realizations of sample size $N=1000$.
First of all, consider the $ARX(1,0)$ process given, for all $n \geq 1$, by
\begin{equation}
\label{ARX1}
X_{n+1}=\theta X_n+U_n+\varepsilon_{n+1}
\hspace{0.5cm}\text{and}\hspace{0.5cm}
\varepsilon_{n+1}=\rho \varepsilon_{n}+V_{n+1}
\end{equation}
where we have chosen $\theta=8/5$ and $\rho=-4/5$ which implies
that $D=18/5$ and the Schur complement $S=16^2/15^2$. This choice has been made in 
order to obtain simple expressions for the matrices $\Lambda$ and $\Sigma$.
One can easily see from \eqref{DEFK} to \eqref{INVLAMBDA} that 
\begin{equation*}
\Lambda=\frac{1}{45}
\begin{pmatrix}
45 & 0 & 0 \\
0 &45 & \!-36\\
0& \!-36 & 80
\end{pmatrix}
\end{equation*}
as well as
\begin{equation*}
\Sigma = \nabla \Lambda^{-1} \nabla^{t}=
\begin{pmatrix}
1&0 \\
0&0
\end{pmatrix} +
\left( \frac{15}{16} \right)^2
\begin{pmatrix}  1&\!-1 \\ \!-1&1 \end{pmatrix}.
\end{equation*}
Figure 1
%\ref{FigASCVG1} 
illustrates the almost sure convergence 
of $\wh{\theta}_{n}$, $\wh{\rho}_n$, $\overline{\rho}_n$
and $\wh{D}_n$. One can see that the almost sure convergence
is very satisfactory.

\begin{figure}[H] 
\vspace{-2cm}
\begin{center}
\includegraphics[width=14cm,height=8cm]{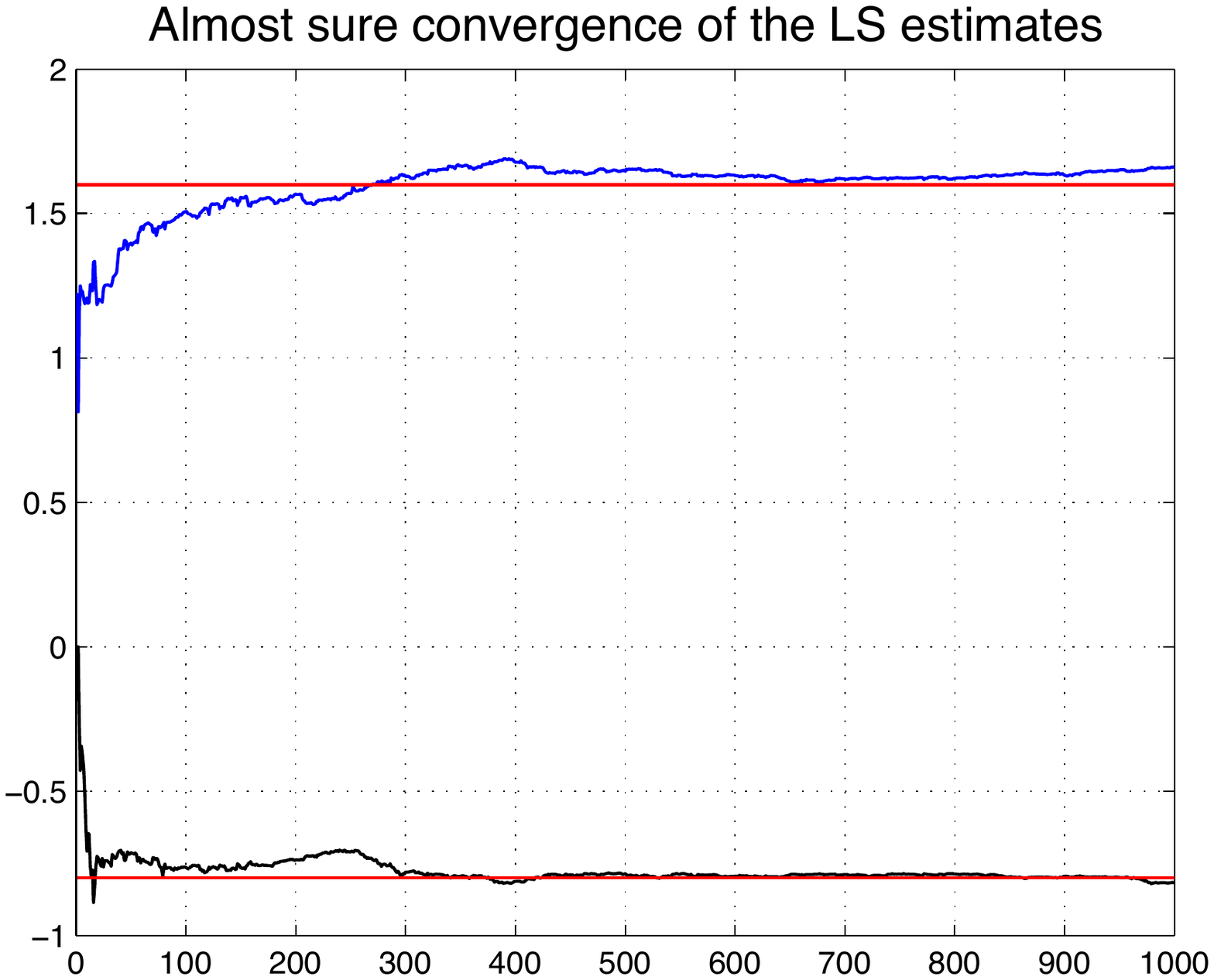}
\end{center}
\par
\vspace{-2cm}
\end{figure}
\begin{figure}[H] 
\label{FigASCVG1}
\begin{center}
\includegraphics[width=14cm,height=8cm]{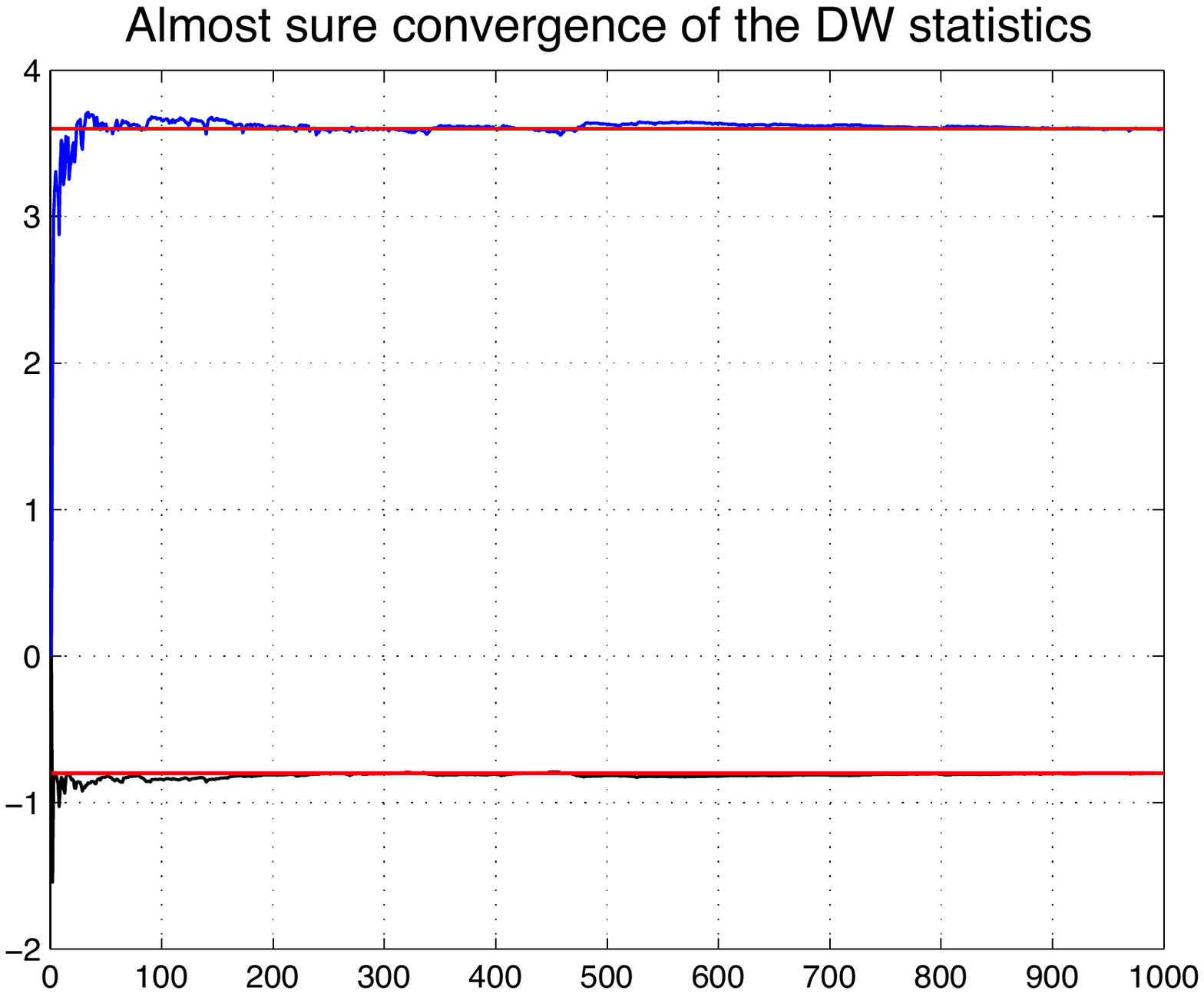}
\end{center}
\vspace{-2cm}
\caption{Almost sure convergence in the particular case $p=1$.}
\vspace{-0.5cm}
\end{figure}

\noindent
We shall now focus our attention to the asymptotic normality. We compare the empirical distributions of the LS estimates
\begin{equation*}
\frac{\sqrt{nS}}{\sqrt{1+S}} \Bigl(\wh{\theta}_{n}-\theta\Bigr)
\hspace{1cm}\text{and}\hspace{1cm}
\sqrt{n S}\Bigl(\wh{\rho}_n - \rho\Bigr)
\end{equation*}
with the standard $\cN(0,1)$ distribution. We proceed in the same way for the Durbin-Watson statistics
\begin{equation*}
\frac{\sqrt{n}}{\tau} \Bigl(\overline{\rho}_n - \rho\Bigr)
\hspace{1cm}\text{and}\hspace{1cm}
\frac{\sqrt{n}}{2\tau}\Bigl(\wh{D}_n-D\Bigr)
\end{equation*}
where $\tau^2$ is given by \eqref{EXPTAU1}.
We use the natural estimates of $S$ and $\tau^2$ by replacing $\rho$ by $\wh{\rho}_n$ and $\overline{\rho}_n$, respectively.
One can see in Figure 2
%\ref{FigCLT1} 
that the approximation by a standard
$\cN(0,1)$ distribution performs pretty well. 
%even for moderate sample size $N=1000$.
These results are very promising in order to built a statistical test based on these
statistics.
\begin{figure}[H] 
\label{FigCLT1}
\vspace{-2cm}
\begin{center}
\begin{tabular}{cc}
\includegraphics[width=7cm,height=11cm]{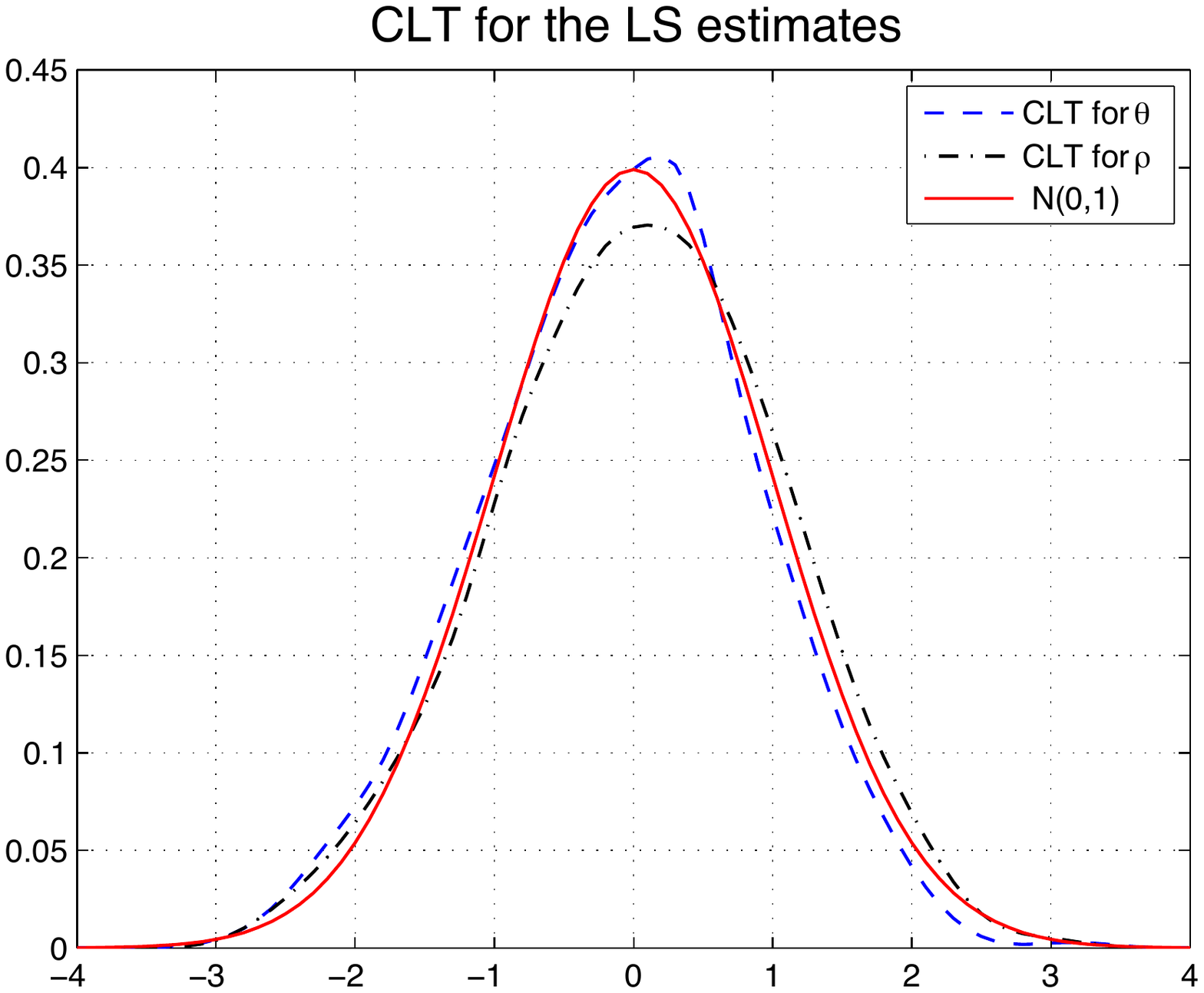}
&
\includegraphics[width=7cm,height=11cm]{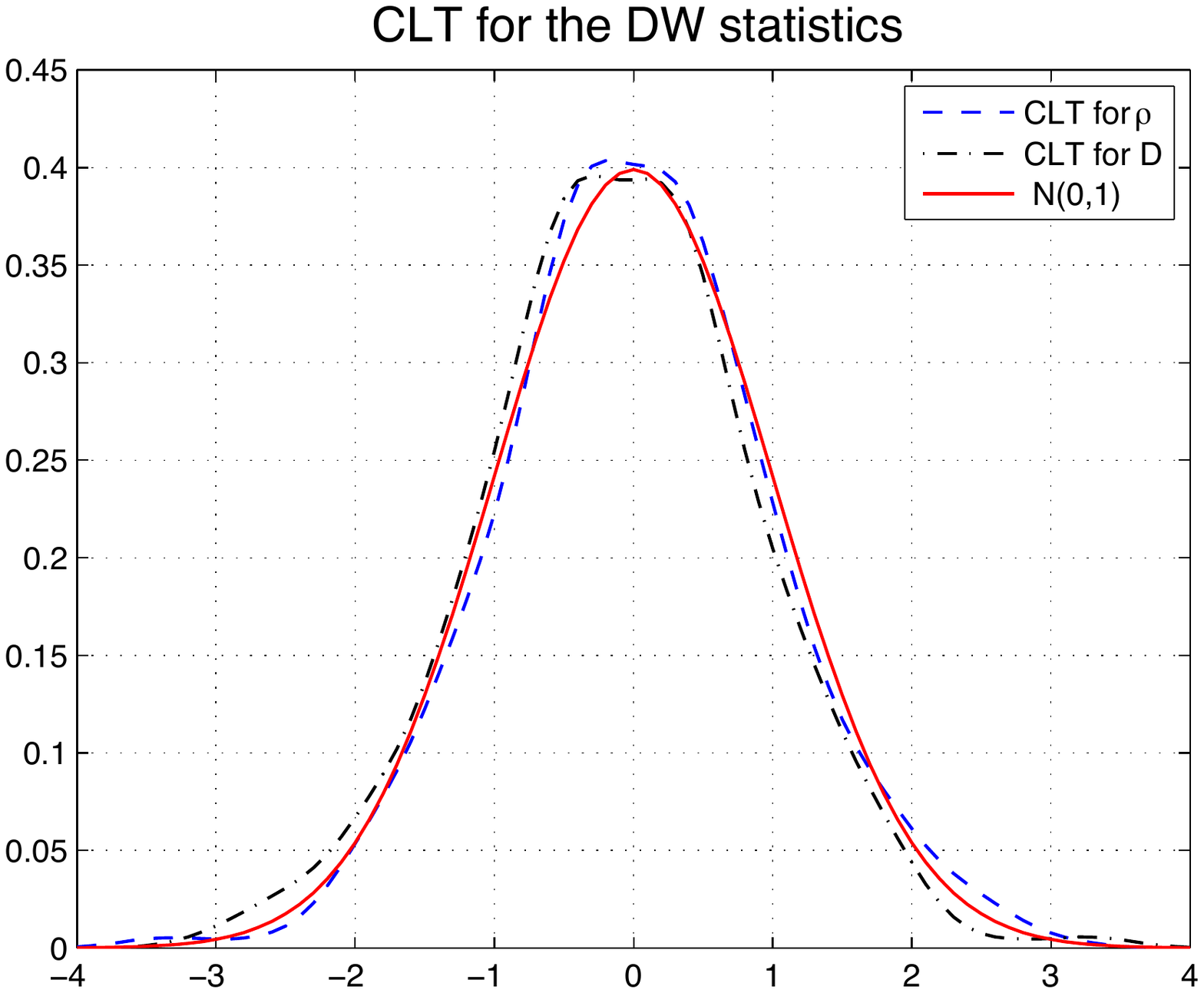}
\end{tabular}
\end{center}
\vspace{-2cm}
\caption{Asymptotic normality in the particular case $p=1$.}
\vspace{-0.5cm}
\end{figure}

\noindent
Next, we are interested in the $ARX(2,0)$ process given, for all $n \geq 1$, by
\begin{equation}
\label{ARX2}
X_{n+1}=\theta_1 X_n+\theta_2 X_{n-1}+U_n+\varepsilon_{n+1}
\hspace{0.5cm}\text{and}\hspace{0.5cm}
\varepsilon_{n+1}=\rho \varepsilon_{n}+V_{n+1}
\end{equation}
where we have chosen $\theta_1=1$,  $\theta_2=4/5$ and $\rho=-9/10$ which leads to
$D=19/5$ and 
$S=9^6/(19\!\times\!10^6)$.
It follows from \eqref{DEFK} to \eqref{INVLAMBDA} that 
\begin{equation*}
\Lambda=\frac{1}{9500}
\begin{pmatrix}
9500 & 0 & 0 & 0 \\
0 & 9500 & 0 & \!-950\\
0 & 0 & 9500 & \!-15295\\
0 & \!-950 & \!-1529 & 51292
\end{pmatrix}.
\end{equation*}
In addition, the diagonal entries of the covariance matrix $\Sigma = \nabla \Lambda^{-1} \nabla^{t}$ are respectively
given by
\begin{equation*}
1+\frac{1}{S}= \frac{721441}{531441}, \hspace{1cm}
1+\rho^2 + \frac{4 \rho^{2} }{S} = \frac{1947541}{656100}, \hspace{1cm}
\frac{1}{S} = \frac{190000}{531441}.
\end{equation*}
Figure 3
%\ref{FigASCVG2} 
shows the almost sure convergence 
of $\wh{\theta}_{n,1}$, $\wh{\theta}_{n,2}$, $\wh{\rho}_n$, $\overline{\rho}_n$
and $\wh{D}_n$ while 
Figure 4
%\ref{FigCLT2} 
illustrates their asymptotic normality. As in the case $p=1$, 
one can observe that the approximation by a standard 
$\cN(0,1)$ distribution works pretty well.

\begin{figure}[H] 
\vspace{-2cm}
\begin{center}
\includegraphics[width=14cm,height=8cm]{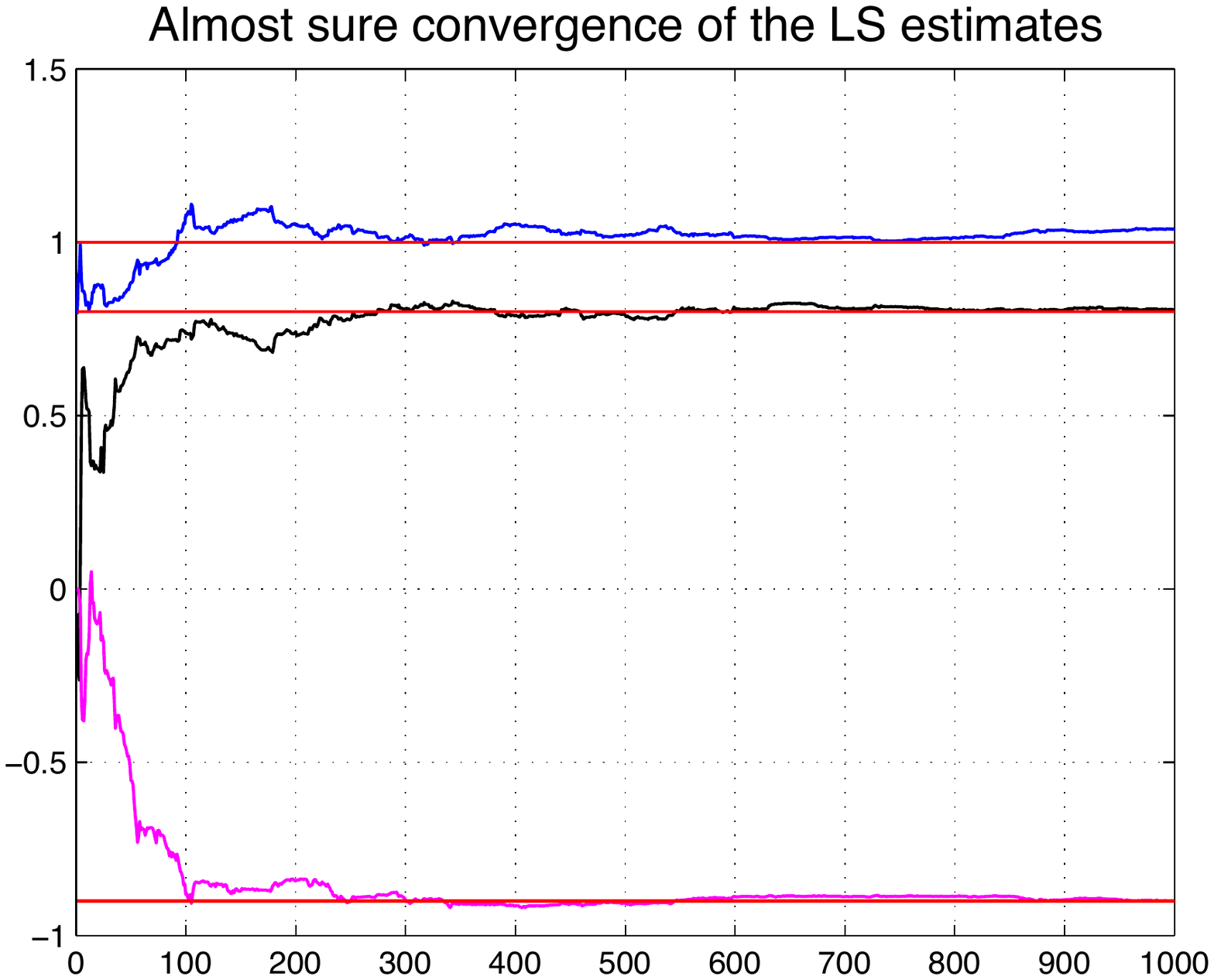}
\end{center}
\vspace{-4cm}
\end{figure}
\begin{figure}[H] 
\label{FigASCVG2}
\begin{center}
\includegraphics[width=14cm,height=8cm]{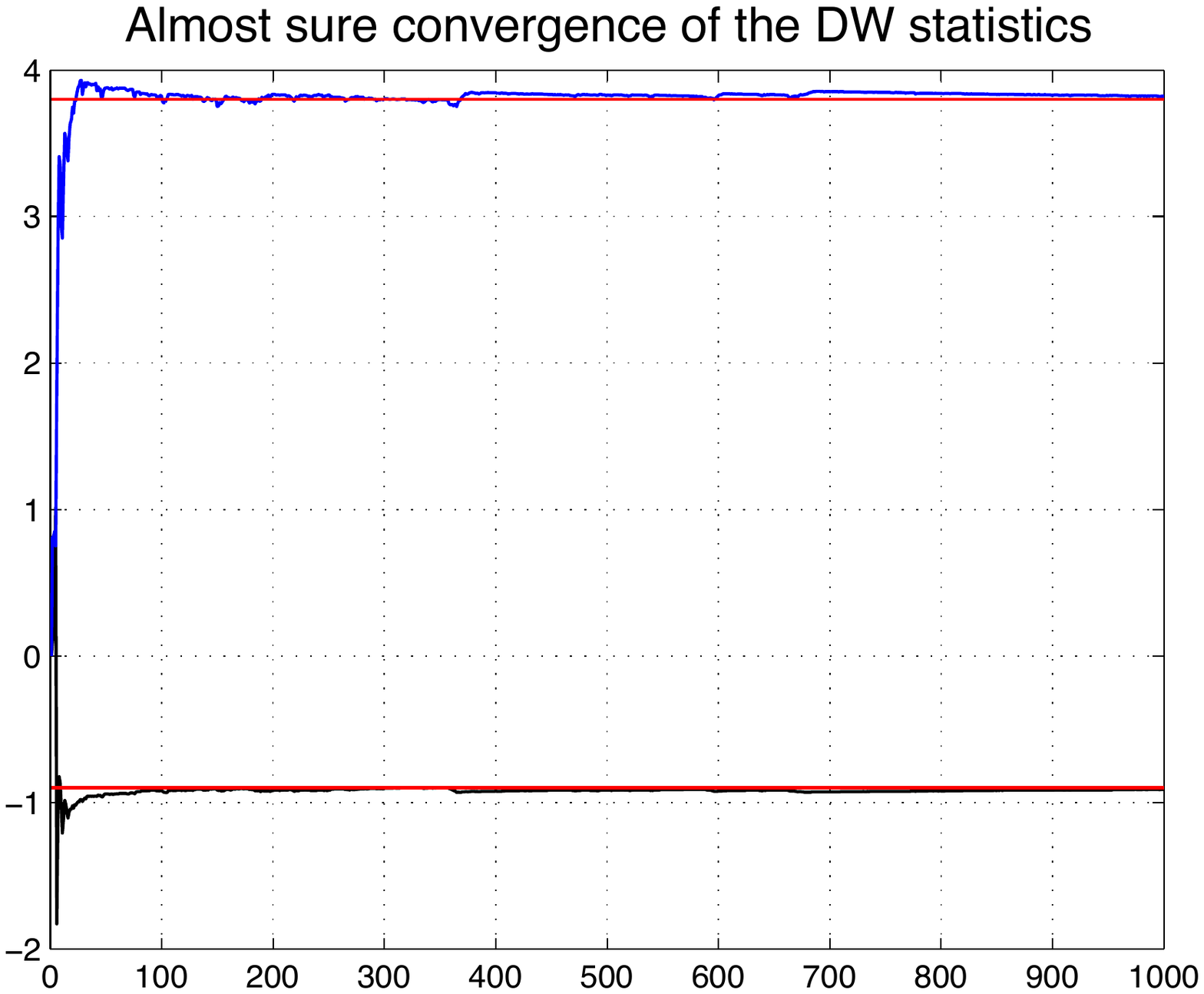}
\end{center}
\vspace{-2cm}
\caption{Almost sure convergence in the particular case $p=2$.}
\vspace{-0.5cm}
\end{figure}

\begin{figure}[H] 
\label{FigCLT2}
\vspace{-2cm}
\begin{center}
\begin{tabular}{cc}
\includegraphics[width=7cm,height=11cm]{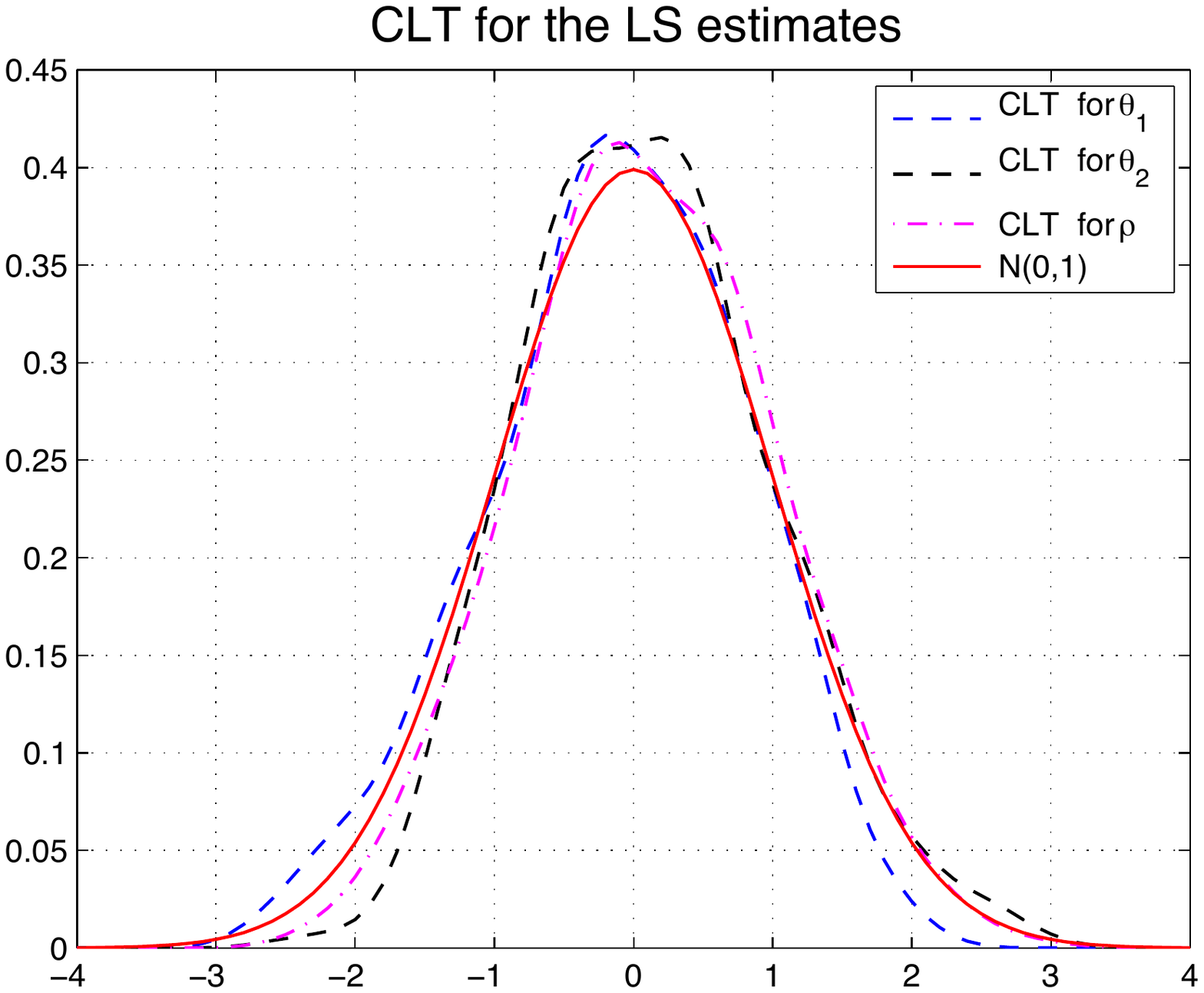}
&
\includegraphics[width=7cm,height=11cm]{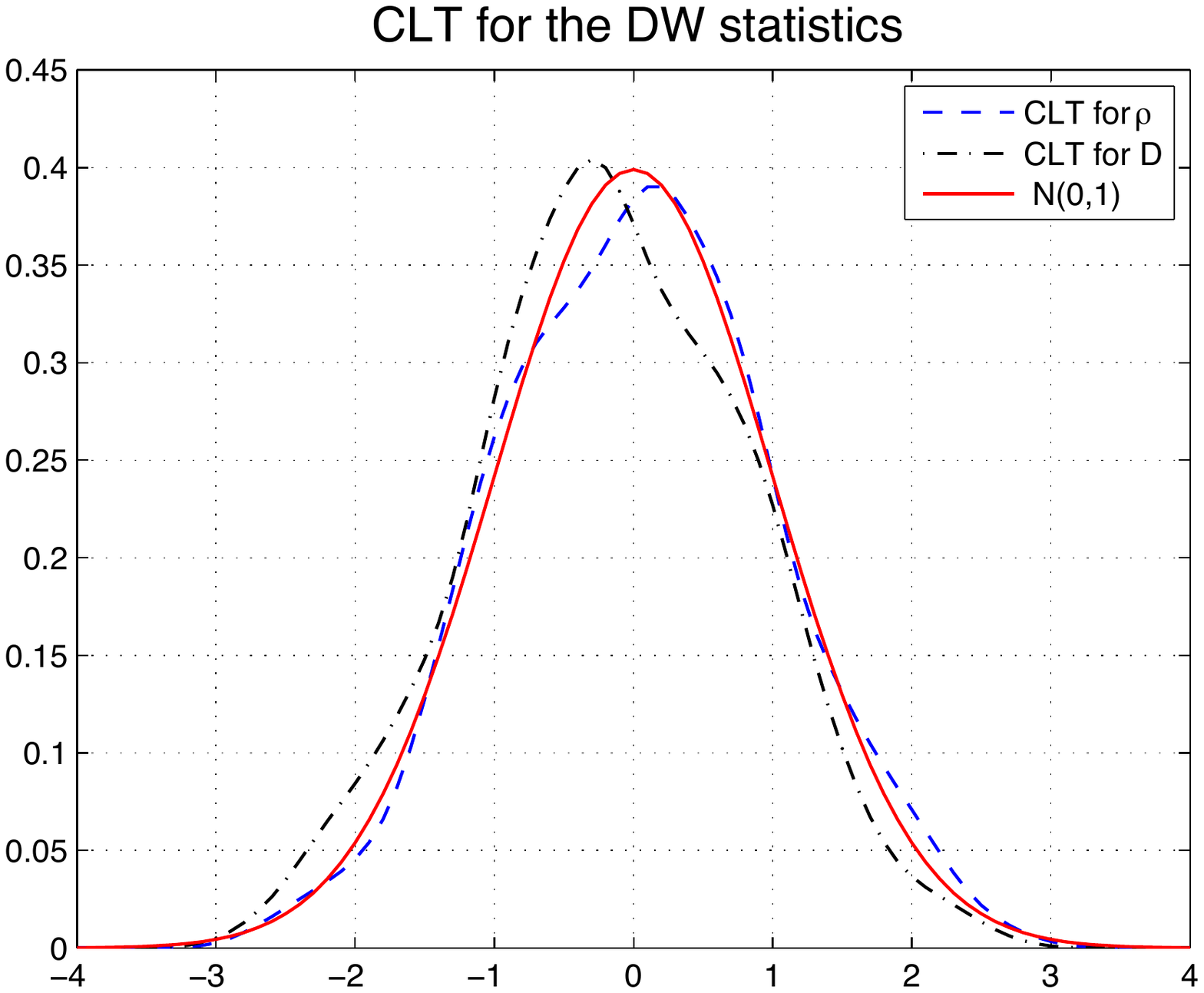}
\end{tabular}
\end{center}
\vspace{-2cm}
\caption{Asymptotic normality in the particular case $p=2$.}
\end{figure}

We shall achieve this section by illustrating the behavior of the Durbin-Watson statistical test.
We wish to test $\cH_0 : ``\rho = \rho_0"$ against $\cH_1 : ``\rho \not= \rho_0"$ 
at $5\%$ level of significance for the ARX processes given by \eqref{ARX1} and \eqref{ARX2}.
More precisely, we compute the frequency for which $\cH_0$ is rejected for different values of $\rho_0$,
\begin{equation*}
\dP\big( \text{rejecting } \cH_0 \hspace{0.1cm} \vert \hspace{0.1cm} \cH_1 \text{ is true} \big)
\end{equation*}
via $500$ realizations of different sample sizes $N=50, 100$ and $1000$.
In Tables 1 and 2, one can appreciate the empirical power of the statistical test
which means that the Durbin-Watson statistic performs very well.

\medskip

\begin{center}
\tiny
\begin{tabular}{|c|c|c|c|c|c|c|c|c|c|c|c|c|}
\hline
\multirow{2}{*}{DW} & \multicolumn{12}{c|}{Values of $\rho_0$} \\
\cline{2-13}
& $-0.9$ & $-0.8$ & $-0.7$ & $-0.6$ & $-0.4$ & $-0.2$ & $0.2$ & $0.4$ & $0.6$ & $0.7$ & $0.8$ & $0.9$  \\
\hline
\multirow{2}{*}{N=50} & 0.20 & 0.02 & 0.12 & 0.38 & 0.79 & 0.95 & 0.99 & 0.99 & 0.99 & 0.99 & 1.00 & 1.00 \\
 & (0.80) & (0.98) & (0.88) & (0.62) & (0.21) & (0.05) & (0.01) & (0.01) & (0.01) & (0.01) & (0.00) & (0.00) \\
\hline
\multirow{2}{*}{N=100} & 0.51 & 0.03 & 0.25 & 0.66 & 0.97 & 0.99 & 1.00 & 1.00 & 1.00 & 1.00 &1.00 & 1.00 \\
 & (0.49) & (0.97) & (0.75) & (0.34) & (0.03) & (0.01) & (0.00) & (0.00) &  (0.00) & (0.00) &  (0.00) & (0.00) \\
\hline
\multirow{2}{*}{N=1000} & 1.00 & 0.05 & 0.99 & 1.00 & 1.00 & 1.00 & 1.00 & 1.00 & 1.00 & 1.00 &1.00 & 1.00 \\
 & (0.00) & (0.95) & (0.01) & (0.01) & (0.00) & (0.00) & (0.00) & (0.00) &  (0.00) & (0.00) &  (0.00) & (0.00) \\
\hline
\end{tabular}\\
\vspace{0.5cm}
\normalsize
Table 1. Durbin-Watson test in the particular case $p=1$ and $\rho=-0.8$.
\end{center}

\vspace{1cm}

\begin{center}
\tiny
\begin{tabular}{|c|c|c|c|c|c|c|c|c|c|c|c|c|}
\hline
\multirow{2}{*}{DW} & \multicolumn{12}{c|}{Values of $\rho_0$} \\
\cline{2-13}
& $-0.9$ & $-0.8$ & $-0.7$ & $-0.6$ & $-0.4$ & $-0.2$ & $0.2$ & $0.4$ & $0.6$ & $0.7$ & $0.8$ & $0.9$  \\
\hline
\multirow{2}{*}{N=50} & 0.06 & 0.17 & 0.52 & 0.76 & 0.92 & 0.96 & 0.99 & 0.99 & 1.00 & 1.00 & 1.00 & 1.00 \\
 & (0.94) & (0.83) & (0.48) & (0.24) & (0.08) & (0.04) & (0.01) & (0.01) & (0.00) & (0.00) & (0.00) & (0.00) \\
\hline
\multirow{2}{*}{N=100} & 0.05 & 0.38 & 0.82 & 0.95 & 0.99 & 0.99 & 1.00 & 1.00 & 1.00 & 1.00 &1.00 & 1.00 \\
 & (0.95) & (0.62) & (0.18) & (0.05) & (0.01) & (0.01) & (0.00) & (0.00) &  (0.00) & (0.00) &  (0.00) & (0.00) \\
\hline
\multirow{2}{*}{N=1000} & 0.05 & 1.00 & 1.00 & 1.00 & 1.00 & 1.00 & 1.00 & 1.00 & 1.00 & 1.00 &1.00 & 1.00 \\
 & (0.95) & (0.00) & (0.00) & (0.00) & (0.00) & (0.00) & (0.00) & (0.00) &  (0.00) & (0.00) &  (0.00) & (0.00) \\
\hline
\end{tabular}\\
\vspace{0.5cm}
\normalsize
Table 2. Durbin-Watson test in the particular case $p=2$ and $\rho=-0.9$.
\end{center}

%%%%%%%%%%%%%%%%%%%%%%%%%%%%%%%%%%%%%%%%%%%%%%%%%%%%%%%%%%%%%%%%%%%%%%%%%%%%%%%%%%%%%%%%%

\section*{Appendix A}

\begin{center}
{\small PROOFS OF THE ALMOST SURE CONVERGENCE RESULTS}
\end{center}

\renewcommand{\thesection}{\Alph{section}} 
\renewcommand{\theequation}
{\thesection.\arabic{equation}} \setcounter{section}{1}  
\setcounter{equation}{0}
%%%%%%%%%%%%%%%%%%%%%%%%%%%%%%%%%%%%%%%%%%%%%%%%%%%%%%%%%%%%%%%%%%%%%%%%%%%%%%%%%%%%%%%%%
\ \\
Denote by $A$ and $B$ the polynomials given, for all $z \in \dC$, by
\begin{equation}
\label{DEFAB}
A(z) = 1 - \sum_{k=1}^{p+1} a_k z^k 
\hspace{1cm}\text{and} \hspace{1cm} 
B(z) = 1 - \rho z
\end{equation}
where $a_1= \theta_1 + \rho$, $a_{p+1}= - \rho \theta_p$ and, for $2\leq k \leq p$,
$$a_k= \theta_k - \rho \theta_{k-1}. $$
The fundamental $\text{ARX}(p+1,1)$ equation given by \eqref{EQFONDA} may be rewritten as
\begin{equation}
\label{POLYFONDA}
A(R) X_n = B(R) U_{n-1} + V_n
\end{equation}
where $R$ stands for the shift-back operator $RX_n = X_{n-1}$. On the one hand, $B(z)=0$
if and only if $z=1/\rho$ with $\rho \neq 0$. Consequently, as $|\rho|<1$, $B$ is clearly causal and
for all $z \in \dC$ such that $|\rho z|<1$,
$$
B^{-1}(z)=\frac{1}{1-\rho z}=\sum_{k=0}^\infty \rho^k z^k.
$$
On the other hand, let $P$ be the polynomial given, for all $z \in \dC$, by
\begin{equation}
\label{DEFP}
P(z)= B^{-1}(z)(A(z) - 1) = \sum_{k=1}^{\infty} p_k z^k. 
\end{equation}
It is not hard to see from \eqref{DEFP} that, for $1\leq k \leq p$, 
$p_k = -(\theta_k + \rho^k)$
while, for all $k \geq p+1$, $p_k=-\rho^k$. Consequently, as soon as $\rho \neq 0$, we deduce from
\cite{BeVa2} that the process $(X_n)$ given by \eqref{POLYFONDA} is strongly controllable.
One can observe that in our situation, the usual notion of controllability is the same as the concept
of strong controllability. To be more precise, the assumption that $\rho \neq 0$ implies that the polynomials
$A-1$ and $B$, given by \eqref{DEFAB}, are coprime. It is exactly the so-called controllability condition.
We refer the reader to \cite{BeVa2} for more details on the links between the notions of controllability and strong controllability.
Finally, we clearly obtain Lemma \ref{L-CVGSN} and Theorem \ref{T-ASCVGVARTHETA} from \eqref{ARXPNEW} together with
Theorem 5 of \cite{BeVa2}.
\hfill
$\mathbin{\vbox{\hrule\hbox{\vrule height1ex \kern.5em\vrule height1ex}\hrule}}$

%%%%%%%%%%%%%%%%%%%%%%%%%%%%%%%%%%%%%%%%%%%%%%%%%%%%%%%%%%%%%%%%%%%%%%%%%%%%%%%%%%%%%%%%%

\section*{Appendix B}

\begin{center}
{\small PROOFS OF THE ASYMPTOTIC NORMALITY RESULTS}
\end{center}

\renewcommand{\thesection}{\Alph{section}} 
\renewcommand{\theequation}
{\thesection.\arabic{equation}} \setcounter{section}{2}  
\setcounter{equation}{0}

%%%%%%%%%%%%%%%%%%%%%%%%%%%%%%%%%%%%%%%%%%%%%%%%%%%%%%%%%%%%%%%%%%%%%%%%%%%%%%%%%%%%%%%%%

Theorem \ref{T-CLTVARTHETA} immediately follows from Theorem 8 of \cite{BeVa2}. We shall now proceed to the proof
of Corollary \ref{C-CLTTHETARHO}. First of all, denote for $0\leq k \leq p-1$, 
$$
s_k(\vartheta)=\sum_{i=1}^{k+1} \rho^{k-i+1} \vartheta_i + \rho^{k} \vartheta_{p+2}
$$
where $\rho= -\vartheta_{p+2}$ and $s_p(\vartheta)= \rho$. In addition, let
\begin{equation}
\label{GCLT}
g(\vartheta) =
 \Delta \vartheta
 =
\begin{pmatrix}
\ s_{0}(\vartheta)  \ \\ 
\ s_{1}(\vartheta)  \ \\ 
\ \vdots  \ \\  
\ s_{p}(\vartheta)  \ 
\end{pmatrix}.
\vspace{1ex}
\end{equation}
One can easily check that the gradient of the function $g$ is given by
\begin{equation}
\label{GRAD}
\nabla g(\vartheta) =
\begin{pmatrix}
1 & 0 & \cdots & \cdots & \cdots & 0 & \xi_0(\theta) \\ 
\rho & 1 & 0 & \cdots & \cdots & 0 & \rho - \xi_1(\theta) \\ 
\rho^2 & \rho & 1 & 0 & \cdots & 0 & \rho^2 -  \xi_2(\theta) \\ 
\cdots & \cdots & \cdots & \cdots & \cdots & \cdots & \cdots \\ 
\rho^{p-1} & \rho^{p-2} & \cdots & \rho & 1 & 0 & \rho^{p-1} -\xi_{p-1}(\theta) \\ 
0 & 0 & \cdots & \cdots & \cdots & 0 & \xi_{p} (\theta)
\end{pmatrix}
\end{equation}
where $\xi_0(\theta)=1$, $\xi_p(\theta)=-1$ and, for all
$1\leq k \leq p-1$, 
$$
\xi_k(\theta) =\sum_{i=1}^{k}\rho^{k-i}\theta_i.
$$
The gradient of $g$ coincides with the matrix $\nabla$ given by \eqref{DEFNABLA}. On the one hand, it follows from
\eqref{INIPARA} and \eqref{GCLT} that
\begin{equation}
\label{DEFG}
g(\vartheta)=
\begin{pmatrix}
\ \theta \ \\
\ \rho \
\end{pmatrix}.
\end{equation}
On the other hand, we already saw from \eqref{CLTVARTHETA} that
\begin{equation}  
\label{CLTVARTHETAB}
\sqrt{n}(\wh{\vartheta}_{n}-\vartheta )\liml
\cN(0,\Lambda^{-1} ).
\end{equation}
Consequently, we deduce from \eqref{DEFG} and \eqref{CLTVARTHETAB} together with the well-known delta method that
\begin{equation*}
\sqrt{n}
\begin{pmatrix}
\ \wh{\theta}_{n} - \theta \ \\
\ \wh{\rho}_{n} - \rho \
\end{pmatrix}
\liml
\cN\left(0, \Sigma \right)
\end{equation*}
where $\Sigma =\nabla \Lambda^{-1} \nabla^{t}$, which completes the proof of 
Corollary \ref{C-CLTTHETARHO}.
\hfill
$\mathbin{\vbox{\hrule\hbox{\vrule height1ex \kern.5em\vrule height1ex}\hrule}}$

%%%%%%%%%%%%%%%%%%%%%%%%%%%%%%%%%%%%%%%%%%%%%%%%%%%%%%%%%%%%%%%%%%%%%%%%%%%%%%%%%%%%%%%%%

\section*{Appendix C}

\begin{center}
{\small PROOFS OF THE DURBIN-WATSON STATISTIC RESULTS}
\end{center}

\renewcommand{\thesection}{\Alph{section}} 
\renewcommand{\theequation}
{\thesection.\arabic{equation}} \setcounter{section}{3}  
\setcounter{equation}{0}

%%%%%%%%%%%%%%%%%%%%%%%%%%%%%%%%%%%%%%%%%%%%%%%%%%%%%%%%%%%%%%%%%%%%%%%%%%%%%%%%%%%%%%%%%

\noindent{\bf Proof of Theorem \ref{T-ASCVGDW}.}
We are now in position to investigate the asymptotic behavior of the Durbin-Watson statistic. 
First pf all, we start with the proof of Theorem \ref{T-ASCVGDW}.
Recall from \eqref{ARXPCF} together with \eqref{RESIDUALN} that the residuals are given, for all $1 \leq k \leq n$, by 
\begin{equation}
\label{NEWRES}
\wh{\veps}_{k}=X_k-U_{k-1}-\wh{\theta}_{n}^{\,t}\varphi_{k-1}=\veps_k - \wt{\theta}_n^{\,t} \varphi_{k-1}
\end{equation} 
where $\wt{\theta}_n = \wh{\theta}_n - \theta$. 
For all $n \geq 1$, denote
\begin{equation*} 
I_{n} =\sum_{k=1}^{n}\wh{\veps}_{k} \wh{\veps}_{k-1}
\hspace{1cm}\text{and}\hspace{1cm}
J_{n} =\sum_{k=0}^{n}\wh{\veps}_{k}^{\, \, 2}.
\end{equation*} 
It is not hard to see that
\begin{eqnarray}
\label{DECOIN}
I_{n} & = & \wh{\veps}_{0}\wh{\veps}_{1} + P_n^{I} - \wt{\theta}_{n}^{\,t}  Q_{n}^{I} + \wt{\theta}_n^{\,t} S_{n-1}^{I}\wt{\theta}_{n}, 
\vspace{1ex}\\
J_{n} & = & \wh{\veps}_{0}^{\,2} + P_n^{J} - 2 \wt{\theta}_{n}^{\,t}  Q_{n}^{J} + \wt{\theta}_n^{\,t} S_{n-1}^{J}\wt{\theta}_{n}
\label{DECOJN}
\end{eqnarray} 
where
\begin{equation*}
P_n^{I}=\sum_{k=2}^n \veps_{k} \veps_{k-1},
\hspace{1cm} 
Q_n^{I}=\sum_{k=2}^n (\varphi_{k-2} \veps_{k}+\varphi_{k-1}\veps_{k-1}),
\hspace{1cm} 
S_n^{I}=\sum_{k=1}^n \varphi_{k}\varphi_{k-1}^{t},
\end{equation*}
and
\begin{equation*}
P_n^{J}=\sum_{k=1}^n \veps_{k}^2,
\hspace{1cm}
Q_n^{J}=\sum_{k=1}^n \varphi_{k-1} \veps_{k},
\hspace{1cm} 
S_n^{J}=\sum_{k=0}^n \varphi_{k}\varphi_{k}^{t}.
\end{equation*}
We deduce from \eqref{COR} that
\begin{equation}
\label{DECPNJ}
(1- \rho^2) P_n^{J}=\rho ^2(\veps_{0}^2 - \veps_{n}^2)+ 2 \rho N_n + L_n
\end{equation}
where
\begin{equation*} 
N_{n} =\sum_{k=1}^{n}\veps_{k-1}V_k
\hspace{1cm}\text{and}\hspace{1cm}
L_{n} =\sum_{k=1}^{n}V_{k}^{2}.
\end{equation*} 
Moreover, we assume that $(V_n)$ has a finite conditional moment of order $a>2$. 
Then, it follows from Proposition 1.3 23 page 25 of \cite{Duflo} that
\begin{equation}
\label{SLLNV}
\lim_{n \rightarrow \infty}\frac{1}{n}\sum_{k=1}^{n}V_k^2=\sigma^2
\hspace{1cm} \text{a.s.}
\end{equation}
In addition, we also have from Corollary 1.3 21 page 23 of \cite{Duflo} that
for all $2\leq b<a$, 
\begin{equation}
\label{CVGV}
\sum_{k=1}^{n} |V_k|^b=\cO(n) \hspace{1cm}\textnormal{a.s.}
\end{equation}
and
\begin{equation}
\label{SUPV}
\sup_{1 \leq k \leq n}  |V_k|=o(n^{1/b}) \hspace{1cm}\textnormal{a.s.}
\end{equation}
However, we clearly obtain from \eqref{COR} that
\begin{equation}
\label{MAJNOISE}
\sup_{1 \leq k \leq n} |\veps_k| \leq \frac{1}{1-|\rho |} \Bigl( |\veps_0| + \sup_{1 \leq k \leq n} | V_{k} | \Bigr)
\end{equation}
and
 \begin{equation}
 \sum_{k=1}^{n} |\veps_{k}|^b \leq 
 \Bigl(  1- | \rho | \Bigr)^{-b} \left(  |\veps_0|^b  +  
 \sum_{k=1}^{n}  |V_{k}|^b \right)
\label{MAJSNOISE} 
 \end{equation}
which of course implies that
\begin{equation}
\label{SUPNOISE}
\sup_{1 \leq k \leq n} |\veps_k|=o(n^{1/b}) \hspace{1cm}\textnormal{a.s.}
\end{equation}
and 
\begin{equation}
 \sum_{k=1}^{n} |\veps_{k}|^b = \cO(n) \hspace{1cm}\textnormal{a.s.}
\label{CVGNOISEB} 
 \end{equation}
In the particular case $b=2$, we find that
\begin{equation}
 \sup_{1 \leq k \leq n}  \veps_{k}^2 = o(n) \hspace{1cm}\text{and}\hspace{1cm}
  \sum_{k=1}^{n} \veps_{k}^2 = \cO(n) \hspace{1cm}\textnormal{a.s.}
\label{RESNOISE2} 
 \end{equation} 
Hereafter,  $(N_n)$ is a locally square-integrable real martingale with predictable quadratic variation
given, for all $n \geq 1$, by
\begin{equation*}
\langle N \rangle_n=\sigma^2\sum_{k=0}^{n-1} \veps_k^2.
\end{equation*}
Therefore, we deduce from \eqref{RESNOISE2} and the strong law of large
numbers for martingales given e.g. by Theorem 1.3.15 page 20 of \cite{Duflo} that
\begin{equation}
\label{CVGMARTN}
\lim_{n\rightarrow \infty} \frac{N_{n}}{n} = 0 \hspace{1cm} \text{a.s.}
\end{equation}
Hence, we obtain from \eqref{DECPNJ} together with \eqref{SLLNV}, \eqref{RESNOISE2} and \eqref{CVGMARTN}
that
\begin{equation}
\label{CVGPNJ}
\lim_{n\rightarrow \infty} \frac{P_{n}^{J}}{n} = \frac{\sigma^2}{1- \rho^2} \hspace{1cm} \text{a.s.}
\end{equation}
Furthermore, convergence \eqref{CVGSN} immediately implies that
\begin{equation}
\label{CVGSNJ}
\lim_{n\rightarrow \infty} \frac{1}{n}S_n^J=\sigma^2 \rI_p \hspace{1cm} \text{a.s.}
\end{equation}
We also obtain from the Cauchy-Schwarz inequality, \eqref{RESNOISE2} and \eqref{CVGSNJ}, that
$$
\parallel Q_n^{J} \parallel = \cO(n) \hspace{1cm}\textnormal{a.s.}
$$
Consequently, we find from the conjunction of \eqref{ASCVGTHETA}, \eqref{DECOJN}, \eqref{CVGMARTN}
and \eqref{CVGSNJ} that
\begin{equation}
\label{CVGJN}
\lim_{n\rightarrow \infty} \frac{J_n}{n}= \frac{\sigma^2}{1- \rho^2} \hspace{1cm} \text{a.s.}
\end{equation}
By the same token, as
\begin{equation}
\label{PNIJ}
P_n^{I}=\rho P_{n-1}^{J}  + N_n + \rho \veps_0^2 -\veps_0 \veps_1,
\end{equation}
it follows from \eqref{CVGMARTN} and \eqref{CVGPNJ} that
\begin{equation}
\label{CVGPNI}
\lim_{n\rightarrow \infty} \frac{P_{n}^{I}}{n} = \frac{\sigma^2 \rho}{1- \rho^2} \hspace{1cm} \text{a.s.}
\end{equation}
which leads via \eqref{DECOIN} to
\begin{equation}
\label{CVGIN}
\lim_{n\rightarrow \infty} \frac{I_n}{n}= \frac{\sigma^2 \rho}{1- \rho^2} \hspace{1cm} \text{a.s.}
\end{equation}
Therefore, we obtain from definition \eqref{NEWLSRHO} together with \eqref{CVGJN} and  \eqref{CVGIN} that
\begin{equation}
\label{CVGROO}
\lim_{n\rightarrow \infty} \overline{\rho}_{n}  = \lim_{n\rightarrow \infty} \frac{I_n}{J_{n-1}} = \rho
\hspace{1cm}\text{a.s.}
\end{equation}
In order to establish the almost sure rate of convergence given by \eqref{ASCVGNRHO}, it is necessary
to make some sharp calculations. We infer from \eqref{DECOIN}, \eqref{DECOJN} and \eqref{PNIJ} that
\begin{equation}
I_{n} - \rho J_{n-1} = N_n - Q_{n} + R_{n}
\label{SRHON}
\end{equation}
where $Q_n=( Q_{n}^{I} - 2 \rho  Q_{n-1}^{J})^{t} \wt{\theta}_{n}$ and
$$
R_n= \wh{\veps}_{0}\wh{\veps}_{1} - \veps_0 \veps_1 + \rho \veps_0^2- \rho\wh{\veps}_{0}^{\,2}  
+ \wt{\theta}_n^{\,t}( S_{n-1}^{I} -\rho  S_{n-2}^{J})\wt{\theta}_{n}.
$$
On the one hand, it follows from convergence \eqref{CVGSN} together with \eqref{ASCVGTHETA} 
and the Cauchy-Schwarz inequality, that
$$
| Q_n | = \cO(\sqrt{n \log n}) \hspace{1cm}
\text{and} \hspace{1cm}
| R_n |= \cO(\log n) \hspace{1cm}
\textnormal{a.s.}
$$
On the other hand, as $\langle N \rangle_n = \cO(n) $ a.s., we deduce from
Theorem 1.3.24 page 26 of \cite{Duflo} related to the almost sure rate of convergence in
the strong law of large numbers for martingales that
$| N_n | = \cO(\sqrt{n \log n}) $ a.s.
Therefore, we can conclude from \eqref{CVGJN} and \eqref{SRHON} that
\begin{equation}
\label{ASCVGNRHOO}  
( \overline{\rho}_{n}-\rho)^{2}= \mathcal{O} \left( \frac{\log n}{n} \right) 
\hspace{0.5cm}\text{a.s.}
\end{equation}
The proof of the almost sure convergence of $\wh{D}_n$ to $D=2(1- \rho)$ immediately follows from \eqref{CVGROO}. 
As a matter of fact, it follows from \eqref{DEFDN} that 
\begin{equation}
\label{DEVDW}
J_{n} \wh{D}_n= 2\big(J_{n-1}-I_n\big)+ \wh{\veps}_{n}^{\, \, 2} -  \wh{\veps}_{0}^{\, \, 2}.
\end{equation}
Dividing both sides of \eqref{DEVDW} by $J_{n-1}$, we obtain that
\begin{equation}
\label{DECODW}
\wh{D}_n= 2(1-f_n)\big(1-\overline{\rho}_n\big)+ g_n
\end{equation}
where 
\begin{equation*} 
f_{n} =\frac{\wh{\veps}_{n}^{\, \, 2}} {J_n}.
\hspace{1cm}
\text{and}
\hspace{1cm}
g_{n} =\frac{\wh{\veps}_{n}^{\, \, 2} - \wh{\veps}_{0}^{\, \, 2}}{J_{n-1}}.
\end{equation*} 
However, convergence \eqref{CVGJN} ensures that $f_{n}$ and $g_{n}$ both tend to zero a.s.
Consequently, \eqref{CVGROO} immediately implies that
\begin{equation}
\label{CVGDOO}
\lim_{n\rightarrow \infty} \wh{D}_{n}  =2(1- \rho)
\hspace{1cm}\text{a.s.}
\end{equation}
The almost sure rate of convergence given by \eqref{ASCVGDW} requires some additional assumption
on $(V_n)$. Hereafter, assume that the noise $(V_n)$ has a finite conditional moment of order $>4$.
We clearly obtain from \eqref{CVGSN}, \eqref{ASCVGTHETA} together with  \eqref{NEWRES} and \eqref{SUPNOISE} with $b=4$ that
\begin{equation}
\label{SUPNOO}
\sup_{1 \leq k \leq n}  \wh{\veps}_{k}^{\, \, 2}=o(\sqrt{n}) + o( \log n) = o(\sqrt{n})\hspace{1cm}\text{a.s.}
\end{equation}
which leads by \eqref{CVGJN} to
\begin{equation}
\label{CVGFN}
f_n = o\left(\frac{1}{\sqrt{n}} \right) 
\hspace{1cm}\text{and} \hspace{1cm}
g_n = o\left(\frac{1}{\sqrt{n}} \right) 
\hspace{1cm}\text{a.s.}
\end{equation}
In addition, it follows from \eqref{DECODW} that
\begin{equation}
\label{DECODDW}
 \wh{D}_{n} - D = -2 (1- f_n) \big( \overline{\rho}_{n} - \rho \big) + 2 (\rho - 1)f_n + g_n
\end{equation}
where $D=2(1-\rho)$. Consequently, we obtain by \eqref{ASCVGNRHOO}  and \eqref{CVGFN} that
\begin{equation}
\Bigl( \wh{D}_{n}-D  \Bigr)^{2} =  \cO\Bigl( ( \overline{\rho}_{n} - \rho )^2 \Bigr) + \cO\bigl(f_n^2\bigr) = \cO \left( \frac{\log n}{n} \right) 
\hspace{1cm}\text{a.s.}
\end{equation}
which achieves the proof of Theorem \ref{T-ASCVGDW}.
\hfill
$\mathbin{\vbox{\hrule\hbox{\vrule height1ex \kern.5em\vrule height1ex}\hrule}}$
%%%%%%%%%%%%%%%%%%%%%%%%%%%%%%%%%%%%%%%%%%%%%%%%%%%%%%%%%%%%%%%%%%%%%%%%%%%%%%%%%%%%%%%%%
\ \vspace{2ex} \\
\noindent{\bf Proof of Theorem \ref{T-CLTDW}.} The proof of Theorem \ref{T-CLTDW} is much more difficult to handle.
We already saw from \eqref{SRHON} that
\begin{equation}
J_{n-1} (\overline{\rho}_{n} - \rho ) = N_n - Q_{n} + R_{n}
\label{SRHOON}
\end{equation}
where the remainder $R_{n}$ plays a negligible role. This is of course not the case for $Q_n=( Q_{n}^{I} - 2 \rho  Q_{n-1}^{J})^{t} \wt{\theta}_{n}$.
We know from \eqref{INIPARA} and \eqref{LSTHETARHO} that
\begin{equation}
\label{MAINDECOCLT}
\begin{pmatrix}
\ \wh{\theta}_{n} - \theta\ \\
\ \wh{\rho}_{n} - \rho \
\end{pmatrix}
 = \wh{\Delta}_{n} \wh{\vartheta}_{n}- \Delta \vartheta = \wh{\Delta}_{n} \Bigl( \wh{\vartheta}_{n} -  \vartheta\Bigr)+
\Bigl( \wh{\Delta}_{n} - \Delta \Bigr) \vartheta.
\end{equation}
One can observe that in the particular case $p=1$, the right-hand side of \eqref{MAINDECOCLT}reduces to
the vector
$$
\Delta \Bigl( \wh{\vartheta}_{n} -  \vartheta\Bigr)
$$
since
$$
\wh{\Delta}_{n} = \Delta = 
\begin{pmatrix}
\ 1 & 0 & 1\ \\
\ 0 & 0 & \!\!-1 \
\end{pmatrix}.
$$
For all $1 \leq k \leq p-1$, denote 
$$
s_n(k)=\sum_{i=0}^{k} \wh{\rho}_{n}^{\,i}\rho^{k-i}.
$$
It is easily check that $\wh{\Delta}_{n} - \Delta$ can be rewritten as
$\wh{\Delta}_{n} - \Delta = ( \wh{\rho}_{n} - \rho)A_n$
where $A_n$ is the rectangular matrix of size $(p+1)\!\times\!(p+2)$ given by
\begin{equation*}
A_n =
\begin{pmatrix}
0 & 0 & \cdots & \cdots & \cdots & 0 & 0 & 0 \\ 
1 & 0 & 0 & \cdots & \cdots & 0 & 0 & 1 \\ 
s_n(1) & 1 & 0 & 0 & \cdots & 0 & 0 & s_n(1) \\ 
\cdots & \cdots & \cdots & \cdots & \cdots & \cdots & \cdots & \cdots \\ 
s_n(p-2) & s_n(p-3) & \cdots & s_n(1) & 1 & 0 & 0 & s_n(p-2)\\ 
0 & 0 & \cdots & \cdots & \cdots & 0 & 0 & 0
\end{pmatrix}.
\end{equation*}
It was already proven that $\wh{\rho}_{n}$ converges almost surely to $\rho$ which 
implies that for all $1 \leq k \leq p-1$,
$$
\lim_{n\rightarrow \infty} s_{n}(k)  =(k+1)\rho^k
\hspace{1cm}\text{a.s.}
$$
It immediately leads to the almost sure convergence of $A_n$ to the matrix $A$ given by
\begin{equation}
\label{DEFA}
A =
\begin{pmatrix}
0 & 0 & \cdots & \cdots & \cdots & 0 & 0 & 0 \\ 
1 & 0 & 0 & \cdots & \cdots & 0 & 0 & 1 \\ 
2 \rho & 1 & 0 & 0 & \cdots & 0 & 0 & 2 \rho \\ 
\cdots & \cdots & \cdots & \cdots & \cdots & \cdots & \cdots & \cdots \\ 
(p-1) \rho^{p-2} & (p-2) \rho^{p-3} & \cdots & 2 \rho & 1 & 0 & 0 & (p-1) \rho^{p-2}\\ 
0 & 0 & \cdots & \cdots & \cdots & 0 & 0 & 0
\end{pmatrix}.
\end{equation}
Denote by $e_{p+2}$ the last vector of the canonical basis of $\dR^{p+2}$. We clearly have
from \eqref{MAINDECOCLT} that
$$\wh{\rho}_{n} - \rho = - e_{p+2}^t \Bigl( \wh{\vartheta}_{n} -\vartheta \Bigr) $$
which implies that
\begin{equation} 
\label{NEWDECOCLT}
\wh{\Delta}_{n} \wh{\vartheta}_{n}- \Delta \vartheta=B_n \Bigl( \wh{\vartheta}_{n} -\vartheta \Bigr)
\end{equation} 
where $B_n=\wh{\Delta}_{n} - A_n \vartheta e_{p+2}^t$. By the same token, let $0_p$ be the null
vector of $\dR^p$ and denote by $\rJ_p$ the rectangular matrix of size $p\!\times\!(p+1)$ given by
$$\rJ_p=\Bigl(\rI_p \ \ 0_p\Bigr).$$
We deduce from \eqref{MAINDECOCLT} and \eqref{NEWDECOCLT} that
\begin{equation}
\label{DECOVARTHETA}
\wt{\theta}_n=\wh{\theta}_{n} - \theta= \rJ_p
\begin{pmatrix}
\ \wh{\theta}_{n} - \theta\ \\
\ \wh{\rho}_{n} - \rho \
\end{pmatrix}
 = \rJ_p \Bigl( \wh{\Delta}_{n} \wh{\vartheta}_{n}- \Delta \vartheta \Bigr)= \rJ_p B_n \Bigl( \wh{\vartheta}_{n} -\vartheta \Bigr).
\end{equation}
We also have from \eqref{LSVARTHETA} that
\begin{equation}
\label{DEFDVARTHETA}
\wh{\vartheta}_{n} -\vartheta=S_{n-1}^{-1}M_n 
\end{equation}
where
$$
M_n=\sum_{k=1}^n \Phi_{k-1}V_k.
$$
Consequently, it follows from \eqref{SRHOON}, \eqref{DECOVARTHETA} and \eqref{DEFDVARTHETA} that
$$
J_{n-1} (\overline{\rho}_{n} - \rho ) = N_n - C_{n}^{t}M_n + R_{n}
$$
where $C_n=S_{n-1}^{-1} B_n^{t} \rJ_p ^{t} T_n$ with $T_n=Q_n^{I} -2 \rho Q_{n-1}^{J}$, which leads to the
main decomposition 
\begin{equation}  
\label{MAINDECOCLTTHETARHO}
\sqrt{n} \begin{pmatrix}
\ \wh{\vartheta}_{n} - \vartheta  \\
\ \overline{\rho}_{n} - \rho
\end{pmatrix}
 =\frac{1}{\sqrt{n}} \cA_n Z_n + \cB_n
\end{equation}
where
\begin{equation*}
Z_n= 
\begin{pmatrix}
M_n  \\
N_n
\end{pmatrix},
\vspace{2ex}
\end{equation*}
\begin{equation*}
\cA_n= n
\begin{pmatrix}
S_{n-1}^{-1} & \ \ 0_{p+2}  \vspace{1ex}\\
J_{n-1}^{-1}C_n^{t} & J_{n-1}^{-1}
\end{pmatrix}
\hspace{1cm}\text{and}\hspace{1cm}
\cB_n= \sqrt{n}
\begin{pmatrix}
\ 0_{p+2} \ \vspace{1ex}\\
 J_{n-1}^{-1}R_n
\end{pmatrix}
\end{equation*}
where $0_{p+2}$ stands for the null
vector of $\dR^{p+2}$. The random sequence $(Z_n)$ is a locally square-integrable
$(p+3)$-dimensional martingale with predictable quadratic variation given, for all $n \geq 1$, by
\begin{equation*}
\langle Z \rangle_n=
\sigma^2 \sum_{k=0}^{n-1}
\begin{pmatrix}
\ \Phi_k \Phi_k^{t} \ & \ \Phi_k \veps_k \ \\
\ \Phi_k^{t}\veps_k \ & \  \veps_k^2 \ 
\end{pmatrix}.
\end{equation*}
We already saw from \eqref{CVGSN} that
\begin{equation}
\label{CVGCPHI}
\lim_{n\rightarrow \infty} \frac{1}{n}\sum_{k=0}^{n}\Phi_k \Phi_k^{t}=\sigma^2 \Lambda \hspace{1cm} \text{a.s.}
\end{equation}
In addition, it follows from \eqref{CVGPNJ} that
\begin{equation}
\label{CVGCVEPS}
\lim_{n\rightarrow \infty} \frac{1}{n}\sum_{k=0}^{n}\veps_k^2=\frac{\sigma^2}{1- \rho^2} \hspace{1cm} \text{a.s.}
\end{equation}
Furthermore, it is not hard to see that
\begin{equation*}
\lim_{n\rightarrow \infty} \frac{1}{n}\sum_{k=1}^{n}X_kV_k=\lim_{n\rightarrow \infty} \frac{1}{n}\sum_{k=1}^{n}X_k\veps_k
=\sigma^2 \hspace{1cm} \text{a.s.}
\end{equation*}
Moreover, we obtain from \eqref{COR} that for all $n \geq p$ and for all $1 \leq \ell \leq p$,
$$\veps_{n} = \rho^\ell \veps_{n-\ell} +  \sum_{i=0}^{\ell-1} \rho^i V_{n-i}.$$
Consequently,
\begin{eqnarray*}
\sum_{k=1}^{n} X_{k-\ell}\veps_{k} & = & 
\sum_{k=1}^{n} X_{k-\ell} \Bigl(\rho^\ell \veps_{k-\ell} +  \sum_{i=0}^{\ell-1} \rho^i V_{k-i}\Bigr),\\
& = & \rho^\ell \sum_{k=1}^{n} X_{k-\ell} \veps_{k-\ell} 
+ \sum_{i=0}^{\ell-1} \rho^i  \sum_{k=1}^{n} X_{k-\ell} V_{k-i},
\end{eqnarray*}
which implies that for all $1 \leq \ell \leq p$,
\begin{equation*}
\lim_{n\rightarrow \infty} \frac{1}{n}\sum_{k=1}^{n}X_{k-\ell}\veps_{k}
=\sigma^2 \rho^\ell\hspace{1cm} \text{a.s.}
\end{equation*}
On the other hand, we infer from \eqref{ARXP} that
$$
\sum_{k=1}^n U_{k-1} \veps_k=\sum_{k=1}^{n}X_k\veps_k - \sum_{k=1}^{n}\veps_k^{2}
-\sum_{i=1}^p \theta_{i} \sum_{k=1}^n X_{k-i} \veps_k.
$$
Hence, we find that
\begin{equation*}
\lim_{n\rightarrow \infty} \frac{1}{n} \sum_{k=1}^{n} U_{k-1}\veps_{k} 
 =  -\sigma^2 \left( \frac{\rho^2}{1- \rho^2} + \sum_{i=1}^p \theta_{i} \rho^{i} \right)
\hspace{1cm} \text{a.s.}
\end{equation*}
Consequently, we obtain that
\begin{equation}
\label{CVGCPHIVEPS}
\lim_{n\rightarrow \infty} \frac{1}{n} \sum_{k=1}^{n} \Phi_{k}\veps_{k} 
 =  \sigma^2 \zeta
\hspace{1cm} \text{a.s.}
\end{equation}
where the vector $\zeta$ is given by
\begin{equation}
\label{DEFZETA}
\zeta
 =
\begin{pmatrix}
\ 1  \ \\ 
\ \rho  \ \\ 
\ \vdots  \ \\  
\ \rho^{p} \ \\
\  \varrho_p \
\end{pmatrix}
\vspace{1ex}
\end{equation}
with
$$
\varrho_p= - \eta \rho^2 - \sum_{i=1}^p \theta_{i} \rho^{i}
\hspace{1cm} \text{and} \hspace{1cm} 
\eta = \frac{1}{1- \rho^2}.
$$
We deduce from \eqref{CVGCPHI}, \eqref{CVGCVEPS} and \eqref{CVGCPHIVEPS} that
\begin{equation}
\label{CVGIPZN}
\lim_{n\rightarrow \infty} \frac{1}{n} \langle Z \rangle_n= \cZ
\hspace{1cm}\text{a.s.}
\end{equation}
where $\cZ$ is the positive-semidefinite symmetric matrix given by
\begin{equation}
\label{DEFCALL}
\cZ=\sigma^4
\begin{pmatrix}
\Lambda & \zeta  \vspace{1ex}\\
\zeta^{t}  & \eta
\end{pmatrix}.
\end{equation}
One can observe that $\cZ$ is not positive-definite as $\det(\cZ)=0$. Nevertheless,
it is not hard to see that $(Z_n)$ satisfies the Lindeberg condition.
Therefore, we can conclude from the central limit theorem for multidimensional martingales 
given e.g. by Corollary 2.1.10 of \cite{Duflo} that
\begin{equation}
\label{CLTZN}
\frac{1}{\sqrt{n}} Z_{n} \liml \cN \Bigl(0, \cZ \Bigr).
\end{equation}
Furthermore, we already saw from \eqref{DEFA} that
\begin{equation*}
\lim_{n\rightarrow \infty} A_{n} 
 =  A
\hspace{1cm} \text{a.s.}
\end{equation*}
which implies that
\begin{equation*}
\lim_{n\rightarrow \infty} B_{n} 
 = \Delta - A \vartheta e_{p+2}^{t}
\hspace{1cm} \text{a.s.}
\end{equation*}
 One can easily check from \eqref{DEFDELTA} and \eqref{DEFA} that
 $$
 \Delta - A \vartheta e_{p+2}^{t}= \nabla
 $$
 where the matrix $\nabla$ is given by \eqref{DEFNABLA}.
Moreover, it follows from the previous calculation that
\begin{equation*}
\lim_{n\rightarrow \infty} \frac{1}{n}T_{n} 
 =  \sigma^2 (1- \rho^2) T \hspace{1cm} \text{a.s.}
\end{equation*}
where the vector $T$ is given by
\begin{equation}
\label{DEFVT}
T
 =
\begin{pmatrix}
1  \\ 
\rho  \\ 
\vdots  \\  
 \rho^{p-1} 
\end{pmatrix}.
%\vspace{1ex}
\end{equation}
Consequently, as 
$C_n=S_{n-1}^{-1} B_n^{t} \rJ_p^{t} T_n$, we obtain from \eqref{CVGSN} that
\begin{equation*}
\lim_{n\rightarrow \infty} C_{n} 
 =  C
\hspace{1cm} \text{a.s.}
\end{equation*}
where
$$
C = (1- \rho^2) \Lambda^{-1} \nabla^{t} \rJ_p ^{t} T.
$$
Hence, we obtain from \eqref{CVGSN} and \eqref{CVGJN} that
\begin{equation}
\label{CVGFINCA}
\lim_{n\rightarrow \infty} \cA_{n} 
 =  \cA
\hspace{1cm} \text{a.s.}
\end{equation}
where
$$
\cA= \sigma^{-2}
 \begin{pmatrix}
\Lambda^{-1} & \ \ 0_{p+2}  \vspace{1ex}\\
(1- \rho^2) C^{t} & (1- \rho^2)
\end{pmatrix}.
$$
In addition, we clearly have from \eqref{CVGJN}  that
\begin{equation}
\label{CVGFINCB}
\lim_{n\rightarrow \infty} \cB_n= \begin{pmatrix}
\ 0_{p+2} \ \vspace{1ex}\\
 0
\end{pmatrix}\hspace{1cm} \text{a.s.}
\end{equation}
Finally, we deduce from the conjunction of \eqref{MAINDECOCLTTHETARHO}, 
\eqref{CLTZN}, \eqref{CVGFINCA}, \eqref{CVGFINCB}, together with
Slutsky's lemma that
\begin{equation*}  
\sqrt{n} \begin{pmatrix}
\ \wh{\vartheta}_{n} - \vartheta  \\
\ \overline{\rho}_{n} - \rho
\end{pmatrix}
 \liml  \cN \Bigl(0, \cA \cZ \cA^{\prime}\Bigr)
\end{equation*}
which leads to 
\begin{equation*}
\sqrt{n}(\overline{\rho}_{n}-\rho)\liml
\cN\left(0,\tau^2\right)
\end{equation*}
where the asymptotic variance $\tau^2$ is given by
$$
\tau^2 = (1-\rho^2)^{2} \Bigl( C^{t} \Lambda C + 2 C^{t} \zeta + \eta \Bigr).
$$
However, one can easily see from \eqref{DEFALPHABETA} and \eqref{DEFGAMMA} that
\begin{eqnarray*}
\tau^2
& = & (1-\rho^2)^{2} \parallel \Lambda^{1/2} \alpha + (1-\rho^2) \Lambda^{-1/2} \nabla^{t} \beta  \parallel^2,  \\
& = & (1-\rho^2)^{2} \parallel \Lambda^{-1/2}( \Lambda \alpha + (1-\rho^2) \nabla^{t} \beta)  \parallel^2, \\
& = & (1-\rho^2)^{2} \parallel \Lambda^{-1/2} \gamma  \parallel^2, \\
& = & (1-\rho^2)^{2}  \gamma^{t} \Lambda^{-1} \gamma,
\end{eqnarray*}
which completes the proof of \eqref{CLTRHOO}. Finally, \eqref{CLTDW} immediately follows
from \eqref{CLTRHOO} together with \eqref{CVGFN} and \eqref{DECODDW}, which achieves 
the proof of  Theorem \ref{T-CLTDW}.
\hfill
$\mathbin{\vbox{\hrule\hbox{\vrule height1ex \kern.5em\vrule height1ex}\hrule}}$
%%%%%%%%%%%%%%%%%%%%%%%%%%%%%%%%%%%%%%%%%%%%%%%%%%%%%%%%%%%%%%%%%%%%%%%%%%%%%%%%%%%%%%%%%
\ \vspace{2ex} \\
\noindent{\bf Proof of Theorem \ref{T-DWTEST}.} The proof of Theorem \ref{T-DWTEST} is straightforward. As a matter
of fact, we already know from \eqref{CLTDW} that under the null hypothesis $\cH_0$,  
\begin{equation}
\label{CLTDWT}
\sqrt{n}(\wh{D}_{n}-D_0)\liml
\cN\left(0,4\tau^2\right)
\end{equation}
where the asymptotic variance $\tau^2$ is given by \eqref{EXPTAU}.
In addition, it follows from \eqref{DEFGAMMATAUNEW} that
\begin{equation}
\label{CVGTAUT}
\lim_{n\rightarrow \infty} \wh{\tau}_{n}^{\,2}= \tau^2
\hspace{1cm} \text{a.s.}
\end{equation}
Hence, we deduce from \eqref{CLTDWT}, \eqref{CVGTAUT} and Slutsky's lemma that
under the null hypothesis $\cH_0$,
\begin{equation*}
\frac{\sqrt{n}}{2\wh{\tau}_n} \left( \wh{D}_{n} - D_{0} \right) \liml \cN ( 0, 1)
\end{equation*}
which obviously implies \eqref{DWTESTH0}. It remains to show that under the alternative hypothesis $\cH_1$,
our test statistic goes almost surely to infinity. Under $\cH_1$, we already saw
from Theorem \ref{T-ASCVGDW} that
$$
\lim_{n\rightarrow \infty} \overline{\rho}_{n} - \rho_0= \rho - \rho_0
\hspace{1cm} \text{a.s.}
$$
and this limit is different from zero. Consequently,
\begin{equation}
\label{CVGTALT}
\lim_{n\rightarrow \infty} n \big( \overline{\rho}_{n} - \rho_0\big)^2= + \infty
\hspace{1cm} \text{a.s.}
\end{equation}
However, we clearly find from \eqref{DECODDW} that
\begin{equation}
\label{DECODWT}
\wh{D}_{n} - D_{0}= -2  \big( \overline{\rho}_{n} - \rho_0\big) +e_n
\end{equation}
where $e_n=-2f_n(1-\overline{\rho}_{n}) + g_n$. Finally, \eqref{CVGTALT} and \eqref{DECODWT}
clearly lead to \eqref{DWTESTH1}, completing the proof of Theorem \ref{T-DWTEST}.
\hfill
$\mathbin{\vbox{\hrule\hbox{\vrule height1ex \kern.5em\vrule height1ex}\hrule}}$

%%%%%%%%%%%%%%%%%%%%%%%%%%%%%%%%%%%%%%%%%%%%%%%%%%%%%%%%%%%%%%%%%%%%%%%%%%%%%%%%%%%%%%%%%


\begin{thebibliography}{00}
\bibitem{Astrom} K. J. Astr\"om and B. Wittenmark. Adaptive Control, 2nd
edition, Addison-Wesley, New York, 1995.

\bibitem{BPDW} B. Bercu and F. Proia. A sharp analysis on the
asymptotic behavior of the Durbin-Watson for the first-order autoregressive process, ESAIM PS,
Vol. 16, 2012.

\bibitem{BeVa2} B. Bercu and V. V\'azquez. A new concept of strong controllability via the Schur complement 
in adapative tracking, Automatica, Vol. 46, pp. 1799-1805, 2010.

\bibitem{Duflo} M. Duflo. Random Iterative Models, Springer Verlag, Berlin, 1997.

\bibitem{Durbin} J. Durbin. Testing for Serial Correlation in Least-Squares Regression When Some of the Regressors are Lagged Dependent Variables. 
Econometrica. Vol. 38 no. 3, pp. 410-421, 1970.

\bibitem{DW1} J. Durbin and G.S. Watson. Testing for serial correlation in Least Squares regression I. Biometrika , Vol. 37, pp. 409-428, 1950.

\bibitem{DW2} J. Durbin and G.S. Watson. Testing for serial correlation in Least Squares regression II. Biometrika , Vol. 38, pp. 159-178, 1951.

\bibitem{DW3} J. Durbin and G.S. Watson. Testing for serial correlation in Least Squares regression III. Biometrika , Vol. 58, pp. 1-19, 1971.
 
\bibitem{Guo1} L. Guo and H. F. Chen. The Astr\"om Wittenmark self-tuning
regulator revisited and ELS-based adaptive trackers, IEEE Trans. Automat.
Control, Vol. 36, pp. 802-812, 1991.

\bibitem{Guo2} L. Guo. Further results on least squares based adaptive
minimum variance control, SIAM J. Control Optim., Vol. 32, pp. 187-212, 1994.

\bibitem{Hall} P. Hall and C. C. Heyde. Martingale limit theory and its
application, Academic Press, New York, 1980.

\bibitem{cristal} R.J. Hill and H.D. Flack. The use of the Durbin Watson d statistic in Rietveld analysis. Journal of Applied Crystallography. Vol. 20 part 5, pp. 356-361, 1987.

\bibitem{Horn} R. A. Horn and C. R. Johnson. Matrix Analysis, Cambridge
University Press, New York, 1990.

\bibitem{Lai} T. L. Lai and C. Z. Wei. Extended least squares and their
applications to adaptive control and prediction in linear systems, IEEE
Trans. Automat. Control, Vol. 31, pp. 898-906, 1986.

\bibitem{Malinvaud} E. Malinvaud. Estimation et pr\'evision dans les mod\'eles \'economiques autor\'egressifs. Review
of the International Institute of Statistics, Vol. 29, 1961.


\bibitem{Nerlove} M. Nerlove and K.F. Wallis. Use of the Durbin Watson statistic in inappropriate situations. Econometrica. Vol. 34 no. 1, pp. 235-238. 1966.


\bibitem{Phillips} P.C.B. Phillips and M. Loretan. The Durbin Watson ratio under infinite variance errors. Journal of Econometrics. Vol. 47, pp. 85-114, 1991.

\bibitem{Sapir} A. Sapir. Use of the Durbin Watson statistic with lagged dependent variables. Metroeconomica. Vol. 29 issue 1-2-3, pp. 169-172, 2006.

\bibitem{Sri} M.S. Srivastava. Asymptotic distribution of Durbin Watson statistic. Economics Letters. Vol. 24, pp. 157-160, 1987. 

\bibitem{Stocker} T. Stocker. On the asymptotic bias of OLS in dynamic regression models with autocorre-
lated errors. Statist. Papers 48 (2007), 81-93.

\bibitem{White} K.J. White. The Durbin Watson test for autocorrelation in nonlinear models. The Review of Economics and Statistics. Vol. 47 no. 2, pp. 370-373, 1992.

\end{thebibliography}
\end{document}